\documentclass[a4paper,11pt, oneside]{amsart}
\usepackage[english]{babel}
\usepackage{amsthm}
\usepackage{amssymb}
\usepackage{amsfonts}
\usepackage{amsmath}
\usepackage{amsthm}
\usepackage[all]{xy}
\usepackage{geometry}
\usepackage{ae,aecompl}
\usepackage{eurosym}
\usepackage{verbatim}
\usepackage{mathrsfs}
\usepackage{caption}
\usepackage{url}
\usepackage{tikz}
\usepackage{pifont}
\usepackage{caption}
\usepackage{hyperref}
\usepackage{calligra}
\usepackage{array}

\newcounter{stepcounter}
\newtheoremstyle{smallcaps}
    {3pt}                    
    {3pt}                    
    {\itshape}                   
    {}                           
    {\sc}                   
    {.}                          
    {.5em}                       
    {}  
    
\newtheoremstyle{smallcapsdef}
    {3pt}                    
    {3pt}                    
    {}                   
    {}                           
    {\sc}                   
    {.}                          
    {.5em}                       
    {}  

\theoremstyle{plain}

\newtheorem{thm}{Theorem}[section]

\newtheorem{lem}[thm]{Lemma}
\newtheorem{prop}[thm]{Proposition}
\newtheorem{cor}[thm]{Corollary}
\newtheorem{conj}[thm]{Conjecture}

\theoremstyle{definition}

\newtheorem{eg}[thm]{Example}
\newtheorem{defn}[thm]{Definition}
\newtheorem{rem}[thm]{Remark}
\newtheorem{remark}[thm]{Remark}

\date{}

\newcommand\bit{\begin{itemize}}
\newcommand\eit{\end{itemize}}
\newcommand\bet{\begin{enumerate}}
\newcommand\eet{\end{enumerate}}
\newcommand\ed{\end{document}}

\DeclareFontFamily{U}{mathx}{\hyphenchar\font45}
\DeclareFontShape{U}{mathx}{m}{n}{
      <5> <6> <7> <8> <9> <10>
      <10.95> <12> <14.4> <17.28> <20.74> <24.88>
      mathx10
      }{}
\DeclareSymbolFont{mathx}{U}{mathx}{m}{n}
\DeclareFontSubstitution{U}{mathx}{m}{n}
\DeclareMathAccent{\widecheck}{0}{mathx}{"71}
\DeclareMathAccent{\wideparen}{0}{mathx}{"75}

\newcommand{\e}{\varepsilon}
\newcommand{\f}{\varphi}

\newcommand\w{\omega}

\newcommand\Om{\Omega}

\newcommand\del{\partial}
\newcommand\adel{\ol{\partial}}

\newcommand\G{\Gamma}

\newcommand\F{{\mathcal F}}

\newcommand\MM{{\mathcal M}}
\newcommand\NN{{\mathcal N}}
\renewcommand{\O}{\mathcal{O}}

\newcommand\exd{\mathrm{d}}

\newcommand\unit{\mathrm{U}}
\newcommand\counit{\mathrm{C}}

\newcommand\id{\mathrm{id}}

\def\qbinom#1#2{\ensuremath{\left[\kern-.3em\left[\genfrac{}{}{0pt}{}{#1}{#2}\right]\kern-.3em\right]_q}}

\newcommand\ol{\overline}

\newcommand\alg{algebra~}

\def\clap#1{\hbox to 0pt{\hss#1\hss}}
\def\mathllap{\mathpalette\mathllapinternal}

\def\mathllapinternal#1#2{%
  \llap{$\mathsurround=0pt#1{#2}$}}

\newsavebox\qModFirst
\newsavebox\qModSecond
\newcommand{\Mod}{\mathrm{Mod}}
\newcommand{\modz}[2]{
  \sbox{\qModFirst}{$#1$}
  \sbox{\qModSecond}{$#2$}
  \ifdim\wd\qModFirst>\wd\qModSecond
    {}^{\phantom{#1}\mathllap{#1}}_{\phantom{#1}\mathllap{#2}}%
  \else
    {}^{\phantom{#2}\mathllap{#1}}_{\phantom{#2}\mathllap{#2}}%
  \fi\mathrm{mod}_0}
\newcommand{\qMod}[4]{{%
    \sbox{\qModFirst}{$#1$}
    \sbox{\qModSecond}{$#2$}
    \ifdim\wd\qModFirst>\wd\qModSecond
      {}^{\phantom{#1}\mathllap{#1}}_{\phantom{#1}\mathllap{#2}}%
    \else
      {}^{\phantom{#2}\mathllap{#1}}_{\phantom{#2}\mathllap{#2}}%
    \fi
    \Mod^{#3}_{#4}}}

\newcommand{\qmod}[4]{{%
      \sbox{\qModFirst}{$#1$}
      \sbox{\qModSecond}{$#2$}
      \ifdim\wd\qModFirst>\wd\qModSecond
        {}^{\phantom{#1}\mathllap{#1}}_{\phantom{#1}\mathllap{#2}}%
      \else
        {}^{\phantom{#2}\mathllap{#1}}_{\phantom{#2}\mathllap{#2}}%
      \fi
  \mathrm{mod}^{#3}_{#4}}}

\newcommand\EE{{\mathcal E}}
\newcommand\FF{{\mathcal F}}

\newcommand{\OO}{\mathcal{O}}

\usepackage{tikz}
\usetikzlibrary{decorations.pathreplacing}

\usepackage{ytableau}

\usepackage[english]{babel}
 \usepackage{mathtools}
 \usetikzlibrary{shapes.geometric}
 
\title[A Noncommutative Complex Structure for $\OO_q(\mathrm{F}_3)$]{Noncommutative Complex Structures for the Full Quantum Flag Manifold of $\OO_q(\mathrm{SU}_3)$} 

\thanks{R\'OB is supported by the GA\v{C}R/NCN grant \emph{Quantum Geometric Representation Theory and Noncommutative Fibrations} 24-11728K. All three authors acknowledge support from COST Action 21109 CaLISTA, supported by COST (European Cooperation in Science and Technology) and HORIZON-MSCA-2022-SE-01-01 CaLIGOLA. AC acknowledges support from  MSCA-DN CaLiForNIA - 101119552. AC and JR are grateful for the hospitality offered by Charles University and by the Institute of Mathematics of the Czech Academy of Sciences.}

\author[A. Carotenuto]{Alessandro Carotenuto}
\address{FaBiT, via San Donato 15, 41127 Bologna, Italy} 
\email{alessandr.carotenut2@unibo.it, acaroten91@gmail.com}

\author[R. \'O Buachalla]{R\'eamonn \'O Buachalla}
\address{Mathematical Institute of Charles University, Sokolovsk\'a 83, Prague, Czech Republic} 
\email{reamonnobuachalla@gmail.com}

\author[J. Razzaq]{Junaid Razzaq}
\address{Department of Mathematics, Piazza di Porta S. Donato, 5, 40126 Bologna, Italy} 
\email{junaid.razzaq2@unibo.it}

\begin{document}

\maketitle


\begin{abstract}
In recent work, Lusztig's positive root vectors (with respect to a distinguished  choice of reduced decomposition of the longest element of the Weyl group) were shown to give a quantum tangent space for every  $A$-series Drinfeld--Jimbo full quantum flag manifold $\OO_q(\mathrm{F}_n)$. Moreover, the associated differential calculus $\Omega^{(0,\bullet)}_q(\mathrm{F}_n)$ was shown to have classical dimension, giving a direct $q$-deformation of the classical anti-holomorphic Dolbeault complex of $\mathrm{F}_n$. Here we examine in detail the rank two case, namely the full quantum flag manifold of $\OO_q(\mathrm{SU}_3)$. In particular, we examine the $*$-differential calculus associated to $\Omega^{(0,\bullet)}_q(\mathrm{F}_3)$ and its non-commutative complex geometry. We find that the number of almost-complex structures reduces from $8$ (that is $2$ to the power of the number of positive roots of $\frak{sl}_3$) to $4$ (that is $2$ to the power of the number of simple roots  of $\frak{sl}_3$). Moreover, we show that each of these almost-complex structures is integrable, which is to say, each of them is a complex structure. Finally, we observe that, due to non-centrality of all the non-degenerate coinvariant $2$-forms, none of these complex structures admits a left $\OO_q(\mathrm{SU}_3)$-covariant noncommutative K\"ahler structure.
\end{abstract}



\section{Introduction}

Constructing a theory of noncommutative geometry for Drinfeld--Jimbo quantum groups is a very important but very challenging problem. Despite numerous significant contributions over the past three decades, this field remains largely under development. Throughout the literature, the essential example has been the celebrated Podle\'s sphere $\OO_q(S^2)$, which serves as a fundamental test for evaluating new ideas. While many important questions remain, the Podle\'s sphere stands out for its relatively well understood noncommutative geometry. This is in sharp contrast to the quantum group $\OO_q(\mathrm{SU}_2)$, where the obstruction posed by the non-existence of a bicovariant differential calculus of classical dimension remains unresolved.

The Podle\'s sphere is the simplest example of a quantum flag manifold. For the last two decades this class of quantum homogeneous spaces has been the focus of intense study, as the noncommutative geometry community has tried to extend its understanding of the Podle\'s sphere to this general class of examples. In particular, attention has focused on those quantum flag manifolds of irreducible type, a special, more tractable, subfamily of the general quantum flag manifolds. This has seen many successes, the most notable being Heckenberger and Kolb's proof that the irreducible quantum flag manifolds admit an essentially unique $q$-deformed covariant de Rham complex. This result directly generalises Podle\'s' construction and classification of differential calculi for $\OO_q(S^2)$. It has been shown that the noncommutative complex and K\"ahler geometry of the Podle\'s sphere \cite{Maj} extends to the irreducible quantum flag setting \cite{MarcoConj,ROBKahler}, as does the Bott--Borel--Weil theorem for $\OO_q(S^2)$ \cite{CDOBBW,BwGrass,KLvSPodles,KKCP2,KKCPN,Maj}. Following Connes' $C^*$-algebraic approach to noncommutative geometry, spectral triples have been constructed for many irreducible quantum flag manifolds \cite{SISSACPn,SISSACP2,DOS1,FredyQuad,MatassaParth,DOW}, extending the construction of D\c{a}browski--Sitarz for the Podle\'s sphere \cite{DSPodles}. Recently, the family of quantum projective spaces have been endowed with the structure of a compact quantum metric space \cite{MK24}, extending the Podle\'s sphere construction  given in \cite{AKK23}.


While many interesting and challenging problems remain in the irreducible setting, the time has now come to examine the non-irreducible situation. The first steps in this direction have already been taken. For example, there is the work of Yuncken and Voigt on the noncommutative geometry of the full quantum flag manifold of $\OO_q(\mathrm{SU}_3)$, which uses a quantum BGG sequence to verify the Baum--Connes conjecture for $U_q(\frak{sl}_3)$ \cite{VoigtYuncken}. More recently, Somberg and the second author constructed an anti-holomorhic Dolbeault complex 
for the $A$-series full quantum flag manifolds using Lusztig's root vectors and extended the Borel--Weil theorem to this setting \cite{ROBPSLusz}. Subsequently, in \cite{Matassa.Equiv.Nil} Matassa introduced an alternative construction of first-order differential calculi for all quantum flag manifolds.


To a certain extent, the full quantum flag manifolds look closer to the Podle\'s sphere than the other irreducible quantum flag manifolds do. For example, their relative Hopf modules are all direct sums of line bundles. However, their differential calculi $\Omega^{(0,\bullet)}_q(\mathrm{F}_n)$ have more noncommutative behaviour than the Heckenberger--Kolb calculi. In particular, their bimodule structure is more involved. This means that one cannot use the monoidal version of Takeuchi's categorical equivalence, a fact that has many important consequences. In the present paper we restrict to the simplest example of a full quantum flag manifold after the Podle\'s sphere, namely $\OO_q(\mathrm{F}_3)$ the full quantum flag manifold of $\OO_q(\mathrm{SU}_3)$, and examine the $*$-differential calculus $\Omega^{\bullet}_q(\mathrm{F}_3)$ associated to $\Omega^{(0,\bullet)}_q(\mathrm{F}_3)$. This offers an accessible and tractable example, making it an excellent starting point for future research in the non-irreducible setting. Just as Podleś' work advanced our understanding of the irreducible setting, $\OO_q(\mathrm{F}_3)$ has the potential to do the same for the non-irreducible case. Indeed, recent  work of Brzezinski and Szymanski \cite{BrzezinskiSzymanski2021} described $\OO_q(\mathrm{F}_3)$ as the total space of a \emph{quantum fibration} over the quantum projective plane, with a Podle\'s sphere fibre.  The authors put this non-principal quantum fibration forward as a motivating example for a proposed theory of noncommutative fibrations with quantum homogeneous fibres. The first steps in this direction were recently taken in \cite{GAPP}, and a large family of new examples were produced.

The first major result of the paper is to show that the maximal prolongation of $\Omega^{1}_q(\mathrm{F}_3)$ has classical dimension. Moreover, we show that it is a Koszul and a Frobenius algebra, and we calculate the Nakayama automorphism $\sigma$. Notably, unlike for the anti-holomorphic sub-calculus, $\sigma$ is not of classical type. 
We next show that, just as in the classical case, there exist covariant connections for the $1$-forms that are not torsion free. This contrasts with the Heckenberger--Kolb case, where we have a unique covariant connection, and moreover, this connection is torsion free. We next classify the left $\OO_q(\mathrm{SU}_3)$-covariant almost-complex structures on $\Omega^{\bullet}_q(\mathrm{F}_3)$. The number of almost-complex structures reduces from $2^{|\Delta^+|}$ (where $\Delta^+$ is a choice of positive roots for $\frak{sl}_3$) to $2^{|\Pi|}$  (where $\Pi$ is the set of associated simple roots). This is because certain almost-complex classical decompositions fail to be bimodule decompositions in the quantum setting, due to the involved bimodule structure of the differential calculus. Thus we see that the classical Weyl group symmetry of the almost-complex structures on $\mathrm{F}_3$ breaks in the quantum setting. Moreover, an almost-complex structure admits a $q$-deformation only if it is integrable. When it does, integrability carries over to the quantum setting, meaning that we do not have any non-integrable noncommutative almost-complex structures. We contrast this with the irreducible quantum flag manifolds that have a unique complex structure, up to identification of opposite complex structures. It is conjectured that this situation generalises to all $A$-series full quantum flag manifolds.


An interesting observation is that the classical nearly K\"ahler structure of $\mathrm{F}_3$ is associated to one of the non-integrable almost complex structures, meaning that we do not have a quantum nearly K\"ahler structure. Another very interesting feature is that the natural quantum analogue of the standard K\"ahler form and in fact all non-degenerate left $\OO_q(\mathrm{SU}_3)$-coinvariant forms, are no longer central. This implies that the calculus does not admit a covariant non-commutative K\"ahler structure, nor a covariant metric in the sense of Beggs and Majid.


\subsection*{Summary of the Paper}

The paper is organised as follows: In \textsection 2 we recall some necessary preliminaries about differential calculi, noncommutative complex structutres, and Drinfeld--Jimbo quantum groups.

In \textsection 3  we present the differential calculus $\Omega^1_q(\mathrm{F}_3)$ as the base of a homogeneous quantum principal bundle and observe that the zero map gives a principal connection. We then calculate the degree two relations of the maximal prolongation. Moreover, we present the associated quantum exterior algebra as a Frobenius algebra and calculate its Nakayama automorphism.

In \textsection 4 we discuss torsion for connections for covariant calculi over quantum homogeneous spaces. We show that, under the assumption that the quantum isotropy subgroup is cosemisimple, a covariant torsion free connection always exists. We also calculate the dimension of the affine space of covariant connections, and torsion free covariant connections,  for $\Omega^1_q(\mathrm{F}_3)$.

In \textsection 5 we classify the left $\OO_q(\mathrm{SU}_3)$-covariant almost-complex structures of $\Omega^1_q(\mathrm{F}_3)$. We show that of the 8 classical almost-complex structures, only four pass to the quantum setting and that all of these are integrable. Finally, we examine how the standard classical K\"ahler form behaves in the quantum setting. We see that there is a three dimensional space of left $\OO_q(\mathrm{SU}_3)$-covariant $2$-forms, and that none of these forms is both non-degenerate and central.

\subsubsection*{Acknowledgements:} We would like to thank Edwin Beggs, Arnab Bhattacharjee, Andrey Krutov, and Ben McKay for many useful discussions.


\section{Preliminaries}

In this section we recall some basic material about covariant differential calculi over Hopf algebras and their associated tangent spaces. We use Sweedler notation, denote by $\Delta$, $\e$ and $S$ the coproduct, counit and antipode of a Hopf algebra respectively. We write $A^{\circ}$ for the dual coalgebra (Hopf algebra) of a (Hopf) algebra $A$, and denote the pairing between $A$ and $A^{\circ}$ by angular brackets. Throughout the paper, all algebras are over $\mathbb{C}$ and assumed to be unital, all unadorned tensor products are over $\mathbb{C}$, and all Hopf algebras are assumed to have bijective antipodes.

\subsection{Quantum Homogeneous Spaces}

We begin by briefly recalling Takeuchi's equivalence for relative Hopf modules, see \cite[Appendix A]{GAPP} for more details. For $A$ a Hopf algebra, we say that a left coideal subalgebra $B \subseteq A$ is a \emph{quantum homogeneous \mbox{$A$-space}} if $A$ is faithfully flat as a right $B$-module and  $B^+A = AB^+$, where $B^+ := \ker(\e|_B)$. We denote by ${}^{A}_B\mathrm{Mod}_B$ the category of two-sided relative Hopf modules, and by ${}^{\pi_B}\mathrm{Mod}_B$ the category of left comodules over the Hopf algebra $\pi_B(A) := A/B^+A$, which we call the \emph{quantum isotropy subgroup}. An equivalence of categories, known as Takeuchi's equivalence, is given by the functor $\Phi:{}^{A}_B\mathrm{Mod}_B \to {}^{\pi_B}\mathrm{Mod}_B$, where $\Phi(\F) = \F/B^+\F$, for any relative Hopf module $\F$, and the functor $\Psi:{}^{\pi_B}\mathrm{Mod}_B \to {}^{A}_B\mathrm{Mod}_B$ is defined using the cotensor product  $\square_{\pi_B}$ over $\pi_B(A)$. A unit for the equivalence is given by $\unit: \F \to (\Psi \circ \Phi)(\F)$, where $\unit(f) =  f_{(1)} \otimes [f_{(0)}]$, and $[f_{(0)}]$ denotes the coset of $f_{(0)}$ in $\Phi(\F)$. For the special case where $B = A$, Takeuchi's equivalence is known as the fundamental theorem of Hopf modules \cite{Monty}.

\subsection{Covariant Differential Calculi over Hopf algebras}

A {\em differential calculus}, or a \emph{dc}, is a differential graded algebra (dga) 
$$
\left(\Om^\bullet \cong \bigoplus_{k \in \mathbb{Z}_{\geq 0}} \Om^k, \exd\right)
$$  
that is generated as an algebra by the elements $a, \exd b$, for $a,b \in \Om^0$. When no confusion arises, we denote the dc by $\Omega^{\bullet}$, omitting the \emph{exterior derivative} $\exd$. We denote the degree of a homogeneous element $\w \in \Om^{\bullet}$ by $|\w|$. For a given algebra $B$, a differential calculus {\em over} $B$ is a differential calculus such that $B = \Om^0$. We say that $\omega \in \Omega^{\bullet}$ is \emph{closed} if $\exd \omega = 0$. If $B'$ is a subalgebra of $B$ then the \emph{restriction} of $\Omega^{\bullet}$ to $B'$ is the dc generated by the elements $\exd b'$, for $b' \in B'$.

A {\em first-order differential calculus}, or a \emph{fodc}, over an algebra $B$ is a pair $(\Om^1,\exd)$, where $\Omega^1$ is a $B$-bimodule and $\exd: B \to \Omega^1$ is a derivation such that $\Om^1$ is generated as a left  (or equivalently right) $B$-module by those elements of the form~$\exd b$, for~$b \in B$. We say that a dc $(\G^\bullet,\exd_\G)$ {\em extends} a fodc $(\Om^1,\exd_{\Om})$ if there exists a bimodule isomorphism $\f:\Om^1 \to \G^1$ such that $\exd_\G = \f \circ \exd_{\Om}$. It can be shown  \cite[\textsection 2.5]{MMF2} that any fodc admits an extension $\Om^\bullet$ which is \emph{maximal} in the sense that there exists a unique differential map from $\Om^\bullet$ onto any other extension of $\Om^1$. We call this extension the {\em maximal prolongation} of $\Om^1$.

For $A$ a Hopf algebra, and $B \subseteq A$ a quantum homogeneous space, a dc $\Omega^\bullet$ over $B$ is said to be \emph{left covariant} if there exists a  coaction $\Delta_L : \Omega^\bullet \to A \otimes \Omega^\bullet$, giving $\Omega^\bullet$ the structure of a left $A$-comodule algebra and with respect to which~$\exd$ is a left $A$-comodule map. (Note that this uniquely defines the coaction $\Delta_L$.) We see that $\Omega^{\bullet}$ is naturally an object in ${}^A_B\mathrm{Mod}_B$. For any covariant dc $\Omega^{\bullet}$ we usually find it notationally convenient to denote $V^{\bullet} := \Phi(\Omega^{\bullet})$.

\subsection{Some Remarks on Quantum Homogeneous Tangent Spaces} \label{subsection:remarksQHTS} 

Let $A$ be a Hopf algebra, and $W \subseteq A^{\circ}$ a Hopf subalgebra of $A^{\circ}$, such that 
$$
B := \, {}^{W}\!A = \Big\{b \in B \,|\, w \triangleright b := b_{(1)} \langle b_{(2)}, w\rangle = \e(w) b\Big\}
$$
is a quantum homogeneous $A$-space, and denote by $B^{\circ}$ its dual coalgebra. A \emph{tangent space} for $B$ is a finite-dimensional subspace $T \subseteq B^{\circ}$ such that $T \oplus \mathbb{C}1$ is a right coideal of $B^{\circ}$ and $WT = T$. For any tangent space $T$, a right $B$-ideal of $B^{+}$ is given by 
\begin{align*}
I := \big\{ y \in B^+ \,|\, X(y) = 0, \textrm{ for all } X \in T \big\},
\end{align*}
meaning that the quotient $V^1 := B^+/I$ is naturally an object in the category ${}^{\pi_B}\mathrm{Mod}_B$. We call $V^1$ the \emph{cotangent space} of $T$. Consider now the object 
\begin{align*}
\Omega^1(B) := A \square_{\pi_B} V^1.
\end{align*}
If $\{X_i\}_{i=1}^n$ is a basis for $T$, and $\{e_i\}_{i=1}^n$ is the dual basis of $V^1$, then the map 
\begin{align*}
\exd: B  \to \Omega^1(B), & & a \mapsto \sum_{i=1}^n (X^+_i \triangleright a) \otimes e_i
\end{align*}
is a derivation, and the pair $(\Omega^1(B),\exd)$ is a left $A$-covariant fodc over $B$. This gives a bijective correspondence between isomorphism classes of tangent spaces and finitely-generated left $A$-covariant fodc \cite{HKTangent}.

In order to give an explicit presentation of the maximal prolongation of a left $A$-covariant fodc $\Omega^1(B)$, we need to recall the notion of a framing calculus: A \emph{framing calculus} for $\Omega^1(B)$ is a left $A$-covariant fodc $\Omega^1(A) \cong A \otimes \Lambda^1$ over $A$ that restricts to $\Omega^1(B)$, such that $V^1$ embeds into $\Lambda^1$, and the image of $V^1$ in $\Lambda^1$ is a right $A$-submodule of $\Lambda^1$. 

Let $\Omega^1(A) \cong A \otimes \Lambda^1$ be a framing calculus for $\Omega^1(B)$, and  let $I \subseteq B^+$ be the ideal corresponding to $\Omega^1(B)$. Consider the subspace
\begin{align*}
I^{(2)} := \left\{\omega(y) := [y_{(1)}^+] \otimes [y_{(2)}^+]  ~ | ~ y \in I \right\} \subseteq V^1 \otimes V^1 \subseteq \Lambda^1 \otimes  \Lambda^1.
\end{align*}
Starting from the tensor algebra $\mathcal{T}(V^1)$ of $V^1$, we construct the $\mathbb{Z}_{\geq 0}$-graded algebra
\begin{align*}
V^{\bullet} :=  \bigoplus_{k \in \mathbb{Z}_{\geq 0}} V^k := \mathcal{T}\big(V^1\big)/\langle I^{(2)} \rangle,
\end{align*}
which we call the \emph{quantum exterior algebra} of $\Omega^1(B)$, and whose multiplication we denote by $\wedge$, motivated by the classical situation. The formula
\begin{align*}
[\exd b_1 \wedge \cdots \wedge \exd b_k] ~ \mapsto  ~ \Big[(\exd b_1)(b_2)_{(1)} \cdots (b_k)_{(1)}\Big]\wedge \bigg(\bigwedge_{i=2}^k \,  \Big[\exd((b_i)_{(i)})(b_{(i+1)})_{(i)} \cdots (b_{(k)})_{(i)}\Big]\bigg)\!,
\end{align*}
determines an isomorphism between $\Phi(\Omega^k(B))$ and $V^{k}$, for $k \in \mathbb{Z}_{\geq 0}$. See Appendix \ref{app:MPFC} for a more detailed discussion of this material.

\subsection{First-Order Differential $*$-Calculi}

A \emph{$*$-differential calculus}, or a \emph{$*$-dc}, over a $*$-algebra $B$ is a differential calculus over $B$ such that the $*$-map of $B$ extends to a (necessarily unique) conjugate-linear involution $*: \Omega^{\bullet}(B) \to \Omega^{\bullet}(B)$ satisfying the identity $\exd(\omega^*) = (\exd \omega)^*$, and for which
\begin{align*}
(\omega \wedge \nu)^* = (-1)^{kl} \nu^* \wedge \omega^*, & & \textrm{ for all } \omega \in \Omega^k, \, \nu \in \Omega^l.
\end{align*}

If we now assume that $A$ is a Hopf $*$-algebra, and that $\Omega^1(A)$ is a left $A$-covariant fodc over $A$, with corresponding tangent space $T,$ then $\Omega^1(A)$ is a $*$-fodc iff $T^* = T$, \cite[Proposition 14.1.2]{KSLeabh} for details. Consider next the case where $T^* \neq T$. Since 
$$
\Delta(X^*) = X^*_{(1)} \otimes X^*_{(2)} \in (T^* \oplus \mathbb{C}) \otimes A^{\circ}, 
$$
$T^*$ is a tangent space, implying in turn that $T+T^*$ is a tangent space. We call $T+T^*$ the \emph{$*$-extension} of $T$.

\subsection{Preliminaries on the Quantum Groups $U_q(\frak{sl}_3)$ and $\OO_q(\mathrm{SU}_3)$}

In this subsection we recall the definition of the Drinfeld--Jimbo quantised enveloping algebra of the simple Lie algebra $\frak{sl}_3$, and the dual quantum coordinate algebra $\OO_q(\mathrm{SU}_3)$. For more details, we direct the reader to \cite{KSLeabh} where the general definitions of the Drinfeld--Jimbo quantised enveloping algebras and their dual quantum coordinate algebras are given. 

The algebra $U_q(\frak{sl}_3)$ is generated by the elements $E_1,E_2,F_1,F_2,K^{\pm 1}_1$, and $K^{\pm 1}_2$, subject to the relations
\begin{align}
K^{\pm 1}_iE_j = q^{\pm a_{ij}} E_jK^{\pm 1}_i, & & K^{\pm 1}_iF_j = q^{\mp a_{ij}} F_jK^{\pm 1}_i,\\
K_i^{\pm 1}K_j^{\pm 1} = K_j^{\pm 1}K_i^{\pm 1}, & & [E_i,F_j]=\delta_{ij}\frac{K_i-K_i^{-1}}{q-q^{-1}},
\end{align}
where $(a_{ij})$ is the Cartan matrix of $\frak{sl}_3$, and the two \emph{Serre relations}
\begin{align}
E_1^2E_2 - (q+q^{-1}) E_1E_2E_1 + E_1E_2^2 = 0, & & F_1^2F_2 - (q+q^{-1}) F_1F_2F_1 + F_1F_2^2 = 0.
\end{align}
A Hopf algebra structure on $U_q(\frak{sl}_3)$ is determined by the coproduct formulae
\begin{align*}
\Delta(E_i) = E_i \otimes K_i + 1 \otimes E_i, & & \Delta(F_i) = F_i \otimes 1 + K^{-1}_i \otimes F_i, & & \Delta(K_i^{\pm 1}) = K_i^{\pm 1} \otimes K_i^{\pm 1},
\end{align*}
and the counit formulae
\begin{align*}
\e(E_i) = \e(F_i) = 0, & & \e(K^{\pm 1}) = 1.
\end{align*}
Moreover, a Hopf $*$-algebra structure is given by 
\begin{align*}
E_i^* = K_iF_i, & & F_i^* = E_iK^{-1}_i, & & K_i^* = K_i.
\end{align*}

Let $u_{ij}$, for $i,j = 1,2,3$, be the elements of $U_q(\frak{sl}_3)^{\circ}$, the Hopf dual of $U_q(\frak{sl}_3)$, defined by the coproduct formula $\Delta(u_{ij}) = \sum_{a=1}^3 u_{ia} \otimes u_{aj}$, and the fact that their only non-zero pairings with the non-unital generators of $U_q(\frak{sl}_3)$ are given by 
\begin{align*}
\langle E_1, u_{21}\rangle = \langle E_2, u_{32}\rangle = \langle F_1, u_{12}\rangle = \langle F_2, u_{23}\rangle = 1, & &  \langle K^{\pm 1}_i, u_{jj} \rangle = q^{\pm(\delta_{i,j-1} - \delta_{ij})}. 
\end{align*}
We denote by $\OO_q(\mathrm{SU}_3)$ the subalgebra of $U_q(\frak{sl}_3)^{\circ}$ generated by the elements $u_{ij}$, for $i,j = 1,2,3$, and call it the \textit{quantum coordinate algebra} of $\mathrm{SU}_3$. It is clear that $\OO_q(\mathrm{SU}_3)$ is a sub-bialgebra of $U_q(\frak{sl}_3)^{\circ}$. In fact, it is a Hopf subalgebra of $U_q(\frak{sl}_3)^{\circ}$ whose antipode satisfies 
\begin{align*}
S(u_{ij}) = (-q)^{i-j}(u_{km}u_{ln} - q u_{kn}u_{lm}),
\end{align*}
where $\{k,l\} := \{1,2,3\}\backslash \{j\}$, and $\{m,n\} := \{1,2,3\}\backslash \{i\}$. Moreover, it is a Hopf $*$-algebra with respect to the $*$-map defined by $(u_{ij})^* = S(u_{ji})$, for all $i, = 1,2,3$. Note that a dual pairing of Hopf algebras between $\OO_1(\mathrm{SU}_3)$ and $U_q(\frak{sl}_3)$ is given by evaluation, and a left action of $U_q(\frak{sl}_3)$ on $\OO_q(\mathrm{SU}_3)$ is given by $X \triangleright a := a_{(1)} \langle a_{(2)}, X\rangle$, for $X \in U_q(\frak{sl}_3)$ and $a \in \OO_q(\mathrm{SU}_3)$.

\subsection{Preliminaries on $\OO_q(\mathrm{F}_3)$ the Full Quantum Flag Manifold of $\OO_q(\mathrm{SU}_3)$} \label{subsection:prelims.full.q.flag}

Consider now the commutative subalgebra of $ U_q(\mathfrak{sl}_3)$ generated by 
   $ \{K^{\pm 1} _i \mid i=1,2\}$ which we denote by $U_q(\mathfrak{h})$.
This is a Hopf subalgebra of $U_q(\mathfrak{sl}_3)$ and we define the \textit{full quantum flag manifold} $\OO_q(\mathrm{F}_3)$ to be the space of invariants of 
\begin{align*}
    \OO_q(\mathrm{F}_3):={}^{U_q(\frak{h})}\OO_q(\mathrm{SU}_3) := \Big\{b \in \OO_q(\mathrm{SU}_3) \,|\, X \triangleright b = \e(X)b, \textrm{ for all } X \in U_q(\frak{h})\Big\}.
\end{align*}
The full quantum flag manifold is a quantum homogeneous space, and a special example of the general class of quantum homogeneous spaces called the quantum flag manifolds, see \cite{GAPP,DijkStok,HKdR} for more details.

The decomposition of $\OO_q(\mathrm{SU}_3)$ into homogeneous components with respect to the action of $U_q(\frak{h})$ is equivalent to having a $\mathcal{P}^+ = \mathbb{Z}^2$ grading
\begin{align} \label{eqn:linebundledecomp}
   \OO_q(\mathrm{SU}_3)= \bigoplus_{\lambda \in \mathcal{P}^+} \mathcal{E}_{\lambda},  
\end{align}
where each $\mathcal{E}_k$ is an \textit{equivariant line bundle} over $\OO_q(\mathrm{F}_3),$ that is an invertible object in the category of relative Hopf bimodules over $\OO_q(\mathrm{F}_3)$. We see that 
$
\EE_0 = \OO_q(\mathrm{F}_3).
$
See \cite[\textsection 5]{GAPP} for further details. Moreover, for the generators of $\OO_q(\mathrm{SU}_3)$ we see that 
\begin{align} \label{eqn:grading.on.gens}
u_{i1} \in \EE_{-\varpi_1}, & & u_{i2} \in \EE_{\varpi_1 - \varpi_2}, & & u_{i3} \in \EE_{\varpi_2}, 
\end{align}
for all $i=1,2,3$, which completely determines the $\mathcal{P}^+$-grading. 

The subalgebra $\mathcal{O}_q(\mathbb{CP}^2)$ of $\OO_q(\mathrm{SU}_3)$ generated by the elements $z^{\alpha_1}_{ij}:= u_{i1}u_{j1}^*$, for $i,j = 1,2,3$, is called the \emph{quantum projective plane} \cite{SISSACP2, DijkStok, KKCP2, Meyer, MMF1}. An isomorphic copy of the quantum projective plane is generated by the elements $z^{\alpha_2}_{ij}:= u_{i3}u_{j3}^*$, for $i,j = 1,2,3$. Both algebras are contained in $\mathcal{O}_q(\mathrm{F}_3)$, and together they generate $\mathcal{O}_q(\mathrm{F}_3)$ as an algebra. Moreover, both subsalgebras are quantum homogeneous spaces, and in fact also examples of quantum flag manifolds, again see \cite[\textsection 5]{GAPP} for further details.

Let us next introduce some notation
\begin{align} \label{eqn:spherical.gens}
    z^{\alpha_1}_i := u_{i1}, & & \textrm{ for } i=1,2,3, & &  z^{\alpha_2}_i := u_{i3}, & & \textrm{ for } i=1,2,3.
\end{align}
The $*$-algebra $\OO_q(S^5)$ generated by the elements $z^{\alpha_1}_i$ is known as the \emph{quantum $5$-sphere}, or the \emph{Vaksmann--Soibelmann $5$-sphere}. An isomorphic $*$-algebra is generated by the elements $z^{\alpha_2}_i$. From \eqref{eqn:grading.on.gens} above, we immediately see that 
\begin{align*}
z_i^{\alpha_1} \in \EE_{-\varpi_1}, & & \overline{z}_i^{\alpha_1} := (z^{\alpha_1})^*  \in \EE_{\varpi_1} & & z_i^{\alpha_2} \in \EE_{\varpi_2}, & & \overline{z}_i^{\alpha_2} := (z_i^{\alpha_2})^* \in \EE_{-\varpi_2}.
\end{align*}
In fact, since these elements generate $\mathcal{O}_q(\mathrm{SU}_3)$ as a $*$-algebra, this gives an alternative complete description of the $\mathcal{P}^+$-grading. 


\section{The Lusztig--de Rham Complex of $\OO_q(\mathrm{F}_3)$}

In this section we present the main result of the paper, a $q$-deformation of the classical de Rham complex of the full flag  manifold of $\mathrm{SU}_3$. We extend the construction of \cite{ROBPSLusz}, where a quantum tangent space $T^{(0,1)}$ was constructed that $q$-deforms the classical anti-holomorphic tangent space of $\mathrm{F}_3$. We start by considering the $*$-extension of $T^{(0,1)}$, and then look at the maximal prolongation of its associated fodc. Finally, the quantum exterior algebra of the maximal prolongation is presented as a Frobenius algebra.

\subsection{The $*$-Extension of $T^{(0,1)}$} \label{subsection:Lusztig.F3}

In what follows, we find it convenient to denote the positive simple generators $E_1$ and $E_2$ by
 $E_{\alpha_1}$ and  $E_{\alpha_2}$ respectively. In addition we will also consider the non-simple root vector
$$
E_{\alpha_1+\alpha_2}:= [E_2,E_1]_{q^{-1}}.
$$
We can understand $E_{\alpha_1+\alpha_2}$ as a Lusztig root vector. Explicitly, recall that the Weyl group of $\frak{sl}_3$ is the symmetric group $S_3$, with standard generators $s_1$ and $s_2$. Then, with respect to the reduced decomposition $s_2s_1s_2$ of the longest element of $\frak{sl}_3$, the element $E_{\alpha_1+\alpha_2}$ is the associated non-simple quantum root vector. This choice of reduced decomposition for $w_0$ induces the following convex ordering (see \cite[Appendix A.1]{ROBPSLusz} for the definition of a convex order) on $\Delta^+:$
\begin{align*}
    \alpha_2 < \alpha_{1}+\alpha_{2} < \alpha_{1},
\end{align*}
and in the following we will write $\beta < \gamma$ if $\beta$ precedes $\gamma$ with respect to this convex ordering.
See \cite[\textsection 6.2]{KSLeabh}, or \cite[Appendix A]{ROBPSLusz}, for a more detailed presentation of Lusztig's root vectors. Consider now the subspace
\begin{align}
 T^{(0,1)} := \mathrm{span}_{\mathbf{C}}\Big\{E_{\alpha_1}, \, E_{\alpha_2}, E_{\alpha_1 + \alpha_2}\Big\}.  
\end{align}
As shown in \cite[Example 3.2.]{ROBPSLusz}, we have the identity 
\begin{align} \label{eqn:LongestECoproduct}
\Delta(E_{\alpha_1 + \alpha_2}) = E_{\alpha_1 + \alpha_2} \otimes K_1K_2 + q^{-1} \nu E_{\alpha_1} \otimes E_{\alpha_2}K_1 + 1 \otimes E_{\alpha_1 + \alpha_2}.      
\end{align}
Thus we see that $T^{(0,1)}$ is a quantum tangent space for $\OO_q(\mathrm{SU}_3)$.

Next we consider the quantum tangent space $T$, the $*$-extension of $T^{(0,1)}$. We denote $T^{(1,0)}:=(T^{(0,1)})^*$ and we see it is spanned by the elements
\begin{align*}
F_{\alpha_1} :=   E_{\alpha_1}^* = K_1F_{1}, & &  F_{\alpha_2} :=   E_{\alpha_2}^* = K_2F_{2}, & & F_{\alpha_1+\alpha_2} := E_{\alpha_1+\alpha_2}^* = q^{-1}K_1K_2[F_1,F_2]_{q^{-1}}.
\end{align*}
Moreover, we can conclude the following coproduct formula 
\begin{align} \label{eqn:LongestFCoproduct}
\Delta(F_{\alpha_1+\alpha_2}) = F_{\alpha_1+\alpha_2} \otimes K_1K_2 + \nu F_{\alpha_1} \otimes  F_{\alpha_2} K_1 + 1 \otimes F_{\alpha_1+\alpha_2}.
\end{align}

Denote the associated $*$-fodc by $\Omega^1_q(\mathrm{SU}_3)$, its cotangent space by $\Lambda^1$, and the basis of $\Lambda^1$ dual to the defining basis of $T$ by 
\begin{align*}
\Big\{e_{\gamma}, \, f_{\gamma} \,|\, \gamma \in \Delta^+ \Big\}.
\end{align*}
The following lemma gives explicit representatives for the cosets of the dual basis. These representatives will be used in a number of calculations in this section.

\begin{lem} \label{lem:uijeij}
It holds that 
\begin{align*}
e_{\alpha_1} =  [u_{21}], & & e_{\alpha_2} =  [u_{32}], & & e_{\alpha_1 + \alpha_2} =  [u_{31}],\\
f_{\alpha_1} =  [q u_{12}], & & f_{\alpha_2} =  [q u_{23}], & & f_{\alpha_1 + \alpha_2} =  [q^2 u_{13}].
\end{align*}
\end{lem}
\begin{proof}
The representatives for the positive basis elements were established in \cite[Lemma 3.6]{ROBPSLusz}. The negative representatives are produced similarly. For example, the calculation
\begin{align*}
\langle F_{\alpha_1}, u_{12}\rangle = \langle K_1F_{1}, u_{12}\rangle = \langle K_1, u_{11} \rangle \langle F_{1}, u_{12}\rangle = q^{-1},
\end{align*}
implies that $qu_{12}$ is a representative for the coset $f_{\alpha_1}$.
\end{proof}

The following proposition determines the right $\OO_q(\mathrm{SU}_3)$-module structure of $\Lambda^1$. The proof is completely analogous to that of \cite[Proposition 3.7]{ROBPSLusz} and so we omit it.

\begin{prop}\label{prop:rightmod}
The right $\OO_q(\mathrm{SU}_3)$-module structure of $\Lambda^1$ is determined by 
\begin{align*}
e_{\gamma} u_{kk} = q^{-( \gamma, \e_k) }e_{\gamma}, & & f_{\gamma} u_{kk} = q^{-( \gamma, \e_k) }f_{\gamma},\\
e_{\alpha_{1}}u_{32} = \nu e_{\alpha_1 + \alpha_2}, & & f_{\alpha_{1}}u_{23} = q^{-1}\nu f_{\alpha_1 + \alpha_2},
\end{align*}
with all other actions by the generators $u_{ij}$ being zero.
\end{prop}

Just as in \cite[\textsection 3.4]{ROBPSLusz}, we now find it instructive to present the non-diagonal actions in the form of a graph. Arrange the basis elements in their natural upper and lower triangular form and draw an arrow from one basis element $b$ to another $b'$ if there exists a generator $u_{ij}$ such that $bu_{ij}$ is a scalar multiple of $b'$:
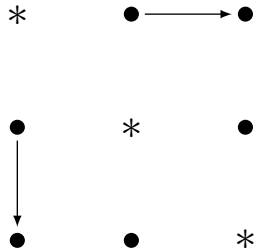
\begin{center}
\begin{tikzpicture}
  \draw (0,1.4) node[scale=1.5] (D1) {*}
  (1.5,-0.1) node[scale=1.5] (D2) {*}
  (3,-1.57) node[scale=1.5] (D3) {*}
        (1.5,1.5) node[circle,fill,inner sep=2pt] (X) {}
        (3,1.5) node[circle,fill,inner sep=2pt] (Y) {}
        (3,0) node[circle,fill,inner sep=2pt] (Z) {}
        (0,0) node[circle,fill,inner sep=2pt] (A) {}
        (0,-1.5) node[circle,fill,inner sep=2pt] (B) {}
        (1.5,-1.5) node[circle,fill,inner sep=2pt] (C) {};
  \draw[->, shorten >=2pt, shorten <=2pt, , >=latex, line width=0.5pt] (A) -- node[right] {} (B);
    \draw[->, shorten >=2pt, shorten <=2pt, , >=latex, line width=0.5pt] (X) -- node[right] {} (Y);
\end{tikzpicture}
\end{center}

\subsection{A Framing Calculus and a Quantum Principal Bundle}

We prove that the pair $(\OO_q(\mathrm{SU}_3),\Omega^1_q(\mathrm{SU}_3))$ gives a quantum principal bundle according to Brzezi\'nski and Majid \cite{BeggsMajid:Leabh,TBSM1}, see also the more recent monograph \cite[Chapter 5]{BeggsMajid:Leabh}. Since principal bundles will not feature elsewhere in the paper, we will not recall here the definition for a general principal comodule algebra. Instead, we recall the well-known fact that, for a quantum homogeneous space $B \subseteq A$, a left $A$-covariant fodc $\Omega^1(A)$ over $A$ gives a quantum principal bundle if and only if $\Omega^1(A)$ is right $\pi_B(A)$-covariant, which is to say, if and only if the following coaction is well-defined: 
\begin{align*}
\Delta_R: \Omega^1(A) \to \Omega^1(A) \otimes \pi_B(A), & & a'\exd a \mapsto a_{(1)}'\exd a_{(1)} \otimes \pi_B(a_{(2)}'a_{(2)}).
\end{align*}
In the following proposition we observe that our fodc on $\OO_q(\mathrm{SU}_3)$ is indeed right covariant.
\begin{prop}
For the quantum homogeneous space $\OO_q(\mathrm{F}_3)$, a quantum principal bundle is given by the pair $(\OO_q(\mathrm{SU}_3),\Omega^1_q(\mathrm{SU}_3))$. 
\end{prop}
\begin{proof}
As discussed in \cite[\textsection 2.2]{ROBPSLusz}, right covariance of the fodc would follow from the identity $TU_q(\frak{h}) = T$, where the action of $U_q(\frak{h})$ on $T$ is given by the right adjoint action. However, since $T$ is spanned by root vectors, this is clear. 
\end{proof}

Denote by $\Omega^1_q(\mathrm{F}_3)$ the restriction of $\Omega^1_q(\mathrm{SU}_3)$ to a left $\OO_q(\mathrm{SU}_3)$-covariant fodc on $\OO_q(\mathrm{F}_3)$, and denote the cotangent space of $\Omega^1_q(\mathrm{F}_3)$ by $V^1$. The following corollary gives an isomorphism between $V^1$ and the left invariant forms of $\Omega^1_q(\mathrm{SU}_3)$. 

\begin{cor} \label{cor:embedding}
An isomorphism in the category of $U_q(\frak{h})$-modules is given by 
\begin{align*} 
V^1 \to \Lambda^1, \quad [b] \mapsto [b].
\end{align*}
\end{cor}
\begin{proof}
The fact that this is an injective module map follows from \cite[Theorem 2.1]{Maj} and the fact that we have a quantum principal bundle. (See also the discussion in \cite[\textsection 4.2]{ROBPSLusz}.) Surjectivity follows from the fact that each basis element of the tangent space $T$ pairs non-trivially with an element of $\OO_q(\mathrm{F}_3)$. Explicitly, the pairings
\begin{align*}
\langle E_{\alpha_1}, z^{\alpha_{1}}_{21} \rangle, & & \langle E_{\alpha_2}, z^{\alpha_{2}}_{32} \rangle, & & \langle E_{\alpha_1+\alpha_2}, z^{\alpha_{1}}_{31} \rangle, & & \langle E_{\alpha_1+\alpha_2}, z^{\alpha_{2}}_{31} \rangle , \\ 
\langle F_{\alpha_1}, z^{\alpha_{1}}_{12} \rangle, & & \langle F_{\alpha_2}, z^{\alpha_{2}}_{23} \rangle, & & \langle F_{\alpha_1 + \alpha_2}, z^{\alpha_{1}}_{13} \rangle, & & \langle F_{\alpha_1 + \alpha_2}, z^{\alpha_{2}}_{13} \rangle.      
\end{align*}
are all non-zero scalars.
\end{proof}

Quantum principal bundles are of particular interest because they allow us to construct affine connections from principal connections on the bundle. In the quantum homogeneous space case $B \subseteq A$, a left $A$-covariant principal connection is equivalent to a left $\pi_B(A)$-comodule projection $\Lambda^1 \to \Lambda^1$ whose kernel is $V^1$, where $V^1$ is the cotangent space of the restriction of $\Omega^1(A)$ to $B$. Following the argument of \cite[Proposition 5.8]{CDOBBW}, we see that Corollary \ref{cor:embedding} implies the existence of a principal connection for the quantum bundle.

\begin{cor}
The zero map on $\Omega^1_q(\mathrm{SU}_3)$ is a left $\OO_q(\mathrm{SU}_3)$-covariant connection for the quantum principal bundle $(\OO_q(\mathrm{SU}_3),\Omega^1_q(\mathrm{SU}_3))$.
\end{cor}

Finally, we note that Corollary \ref{cor:embedding} also implies that our fodc on $\OO_q(\mathrm{SU}_3)$ is a framing calculus for $\OO_q(\mathrm{F}_3)$.

\begin{cor}
It holds that $\Omega^1_q(\mathrm{SU}_3)$ is a framing calculus for $\Omega^1(\mathrm{F}_3)$.
\end{cor}
\begin{proof}
Since the embedding in Corollary \ref{cor:embedding} above is surjective, the image of $V^1$ is obviously a right $\OO_q(\mathrm{SU}_3)$ submodule of $\Lambda^1$.
Thus $\Omega^1_q(\mathrm{SU}_3)$ is a framing calculus, as claimed.
\end{proof}

\subsection{The Higher Forms}

In this section we calculate the degree two relations of the maximal prolongation of the fodc $\Omega^1_q(\mathrm{F}_3)$, and using an application of Bergman's diamond lemma show that the quantum exterior algebra has classical dimension.

\begin{thm}\label{thm:thereal}
A full set of relations for the quantum exterior algebra $V^{\bullet}$ is given by three sets of identities: The first set is given by 
\begin{align*}
 e_{\gamma} \wedge e_{\beta} = -q^{(\beta,\gamma)} e_{\beta} \wedge e_{\gamma}, & & 
 f_{\gamma} \wedge f_{\beta} = -q^{-(\beta,\gamma)} f_{\beta} \wedge f_{\gamma}, & &  \textrm{ for all } \beta \leq \gamma \in \Delta^+, 
\end{align*}
the second set is given by 
\begin{align*}
   e_{\gamma} \wedge f_{\beta} = -q^{(\beta, \gamma)} f_{\beta} \wedge e_{\gamma}, & &  \textrm{ for all } \beta \neq \gamma \in \Delta^+ , \textrm{ or for } \beta = \gamma = \alpha_1+\alpha_2,
\end{align*}
and the third set is given by  the two identities 
\begin{align*}
 & e_{\alpha_1} \wedge f_{\alpha_1} = -q^{2} f_{\alpha_1} \wedge e_{\alpha_1} -  \nu f_{\alpha_1+\alpha_2} \wedge e_{\alpha_1+\alpha_2},  \\
 & e_{\alpha_2} \wedge f_{\alpha_2} = -q^{2} f_{\alpha_2} \wedge e_{\alpha_2}  + \nu f_{\alpha_1+\alpha_2} \wedge e_{\alpha_1+\alpha_2}. & 
\end{align*}
\end{thm}\label{thm:THERELATIONS}
\begin{proof}
Using the description of the right $\OO_q(\mathrm{SU}_3)$-module structure of $\Lambda^1$ given above, one can observe the following set of identities, analogous to those in Lemma \ref{lem:uijeij}: 
\begin{align*}
[S(u_{21})] = -q e_{\alpha_1}, & & [S(u_{32})] = -q e_{\alpha_2}, & & [S(u_{31})] = -q e_{\alpha_1+\alpha_2},\\
[S(u_{12})] = -q^{-2} f_{\alpha_1}, & & [S(u_{23})] = -q^{-2} f_{\alpha_2}, & & [S(u_{13})] = -q^{-5} f_{\alpha_1 + \alpha_2}.
\end{align*}
Moreover, we have the second set of identities, analogous to those in Proposition \ref{prop:rightmod}:
\begin{align}
e_{\gamma} S(u_{kk}) = q^{(\gamma, \varepsilon_{k})} e_{\gamma}, & & f_{\gamma} S(u_{kk}) = q^{(\gamma, \varepsilon_{k})} f_{\gamma}, \label{antipode-action1}\\
e_{\alpha_{1}}S(u_{32}) = -\nu e_{\alpha_1 + \alpha_2}, & & f_{\alpha_{1}}S(u_{23}) = -q^{-3}\nu f_{\alpha_1 + \alpha_2}, \label{antipode-action2}
\end{align}
with all other actions by the antipoded generators $S(u_{ij})$ being zero.

From these identities, we can now see that the following set is the dual basis of $V^{1}$:
\begin{align*}
e_{\alpha_1}=q^{-1}[z_{21}^{\alpha_1}], & & e_{\alpha_2}=-q^{-1}[z_{32}^{\alpha_2}], & & e_{\alpha_1+\alpha_2}=q^{-1}[z_{31}^{\alpha_1}]=-q^{-1}[z_{31}^{\alpha_2}]\\
f_{\alpha_1}=-q^{2}[z_{12}^{\alpha_1}], & & f_{\alpha_2}=q^{2}[z_{23}^{\alpha_2}], & & f_{\alpha_1+\alpha_2}=-q^{5}[z_{13}^{\alpha_1}]=q^{3}[z_{13}^{\alpha_2}].
\end{align*}

We next introduce a set of generators for the ideal of the tangent space from the description of the right $\OO_q(\mathrm{SU}_3)$-module structure of $\Lambda^1$ given above. We divide the set of generators according to their polynomial degree. To do so we find it convenient to introduce the subset of $\mathbb{Z}_{> 0}^3$
$$
B := \Big\{ (1,2,1), (1,1,2), (1,3,1), (1,1,3), (2,3,2), (2,2,3), (2,3,1), (2,1,3) \Big\}.
$$
Consider now the degree one polynomials
\begin{align*}
    G_1 := \Big\{z^{\alpha_i}_{ab} \, | \, (i,a,b) \notin B \Big\} \cup \Big\{z^{\alpha_1}_{31} + z^{\alpha_2}_{31}, \, q^2z^{\alpha_1}_{13} + z^{\alpha_2}_{13}\Big\}
\end{align*}
Next consider the quadratic polynomials
\begin{align*}
G_2 := & \Big\{z^{\alpha_i}_{kl}(z^{\alpha_p}_{ab})^+ \,|\, (i,k,l) \in B \backslash \{(1,2,1), \, (2,3,2)\}, \, p=1,2, \, a,b = 1,2,3\Big\},\\
G_3 := &  \Big\{z^{\alpha_i}_{kl}(z^{\alpha_p}_{ab})^+ \,|\, (i,k,l) \in B, \, (p,a,b) \neq (2,3,2),(2,2,3) \Big\},\\
G_4 := & \Big\{z^{\alpha_1}_{21}z^{\alpha_2}_{32} - \nu z^{\alpha_2}_{31}, \, z^{\alpha_1}_{12}z^{\alpha_2}_{23} - \nu z^{\alpha_1}_{13} \Big\}.
\end{align*}
%
%
Collecting these elements together gives us our proposed set of generators 
$$
G := G_1 \cup G_2 \cup G_3 \cup G_4.
$$
Indeed, since it is clear that 
$$
\mathrm{dim}\left(\OO_q(\mathrm{F}_3)^+/\langle G \rangle \right) \leq 6,
$$
where $\langle G \rangle$ is the right ideal of $\OO_q(\mathrm{F}_3)^+$ generated by the elements of $G$, we see that $G$ gives a full set of generators.


Calculating the action of the map $\omega$ on these generators is now a routine calculation, as explicitly presented in \cite[Proposition 5.8]{MMF2} for the case of quantum projective space, and in \cite[\textsection 3.3]{ROBPSLusz} for the anti-holomoprhic complex of the $A$-series full quantum flag manifold. As an example, we take the generator $z_{22}^{\alpha_1}$, and note that 
\begin{eqnarray}
\omega(z_{22}^{\alpha_1}) &=& \sum_{a,b}[u_{2a}S(u_{b2})] \otimes [z_{ab}^{\alpha_1}] \nonumber\\
&=& [u_{22}S(u_{12})] \otimes [z_{21}^{\alpha_1}]+ [u_{23}S(u_{12})] \otimes [z_{31}^{\alpha_1}] \nonumber\\
&&+[u_{21}S(u_{22})] \otimes [z_{12}^{\alpha_1}]+[u_{21}S(u_{32})] \otimes [z_{13}^{\alpha_1}].\label{w(g)} 
\end{eqnarray} 
Using identities in \ref{antipode-action1} and \ref{antipode-action2}, we compute that:
\begin{align*}
[u_{22}S(u_{12})]=-q^{-2}f_{\alpha_1},\, & & [u_{23}S(u_{12})]=0,~~~~~\,~~~~~\\
[u_{21}S(u_{22})]=q^{-1}e_{\alpha_1},~~~ & &  [u_{21}S(u_{32})]=-\nu e_{\alpha_1 + \alpha_2}.
\end{align*}
Substituting these values in Eq. \ref{w(g)} we get:
$$w(z_{22}^{\alpha_1})=-q^{-1}f_{\alpha_1}\otimes e_{\alpha_1}-q^{-3}e_{\alpha_1}\otimes f_{\alpha_1}+q^{-5}\nu e_{\alpha_1+\alpha_2} \otimes f_{\alpha_1+\alpha_2}.$$
Therefore, we get the relation
$$e_{\alpha_1}\wedge f_{\alpha_1}=-q^{2}f_{\alpha_1}\wedge e_{\alpha_1}+q^{-2}\nu e_{\alpha_1+\alpha_2}\wedge f_{\alpha_1+\alpha_2}.$$
Continuing as such gives us the other claimed relations. Finally, we observe the description of the right $\OO_q(\mathrm{SU}_3)$-module structure of $\Lambda^1$ given in Proposition \ref{prop:rightmod} implies that the relations form a right $\OO_q(\mathrm{F}_3)$-submodule of $V^1 \otimes V^1$. Thus they give a full set of relations.
\end{proof}

\begin{cor}
For  $k=1, \dots, 6 = |\Delta|$, a basis of $V^{k}$ is given by 
\begin{align*}
\Big\{  e_{\gamma_1} \wedge \cdots \wedge e_{\gamma_a} \wedge f_{\gamma_1} \wedge \cdots \wedge f_{\gamma_b} \,|\, \gamma_1 < \cdots < \gamma_k \in \Delta^+ \Big\}. 
\end{align*}
In particular, it holds that
\begin{align*}
\mathrm{dim}\Big(V^{k}\Big) = \binom{\,|\Delta|\,}{k}, & & \textrm{ and } & & \mathrm{dim}\Big(V^{\bullet}\Big) = 2^{|\Delta|}.
\end{align*}
\end{cor}
\begin{proof}
It is clear from the set of relations given in Theorem \ref{thm:THERELATIONS} that the proposed basis is a spanning set. To prove that its elements are linearly independent, let $\textbf{M}$ be the set of monomials in the basis elements of $V^{1}$, and let $S_{\Delta^+} :=(W_{\Delta^+},f_{\Delta^+})$ be the reduction system in the tensor algebra $\mathcal{T}(V^{1})$ given by the family of pairs
\begin{align*}
& \big(e_{\gamma} \otimes e_{\beta} ,\; -q^{(\beta,\gamma)} e_{\beta} \otimes e_{\gamma}\big),\quad
 \big(f_{\gamma} \otimes f_{\beta},-q^{-(\beta,\gamma)} f_{\beta} \otimes f_{\gamma}\big), & &  \textrm{ for all } \beta \leq \gamma \in \Delta^+,
\\&(e_{\gamma} \otimes f_{\beta}, \; -q^{(\beta,\gamma)} f_{\beta} \otimes e_{\gamma}), & & \textrm{ for all } \beta \neq \gamma \in \Delta^+ , \\
& & & \textrm{ \, or for } \beta = \gamma = \alpha_1+\alpha_2, 
\end{align*}
and together with the two pairs
\begin{align*}
&\big(e_{\alpha_1} \otimes f_{\alpha_1} ,\; -q^{2} f_{\alpha_1} \otimes e_{\alpha_1} -  \nu f_{\alpha_1+\alpha_2} \otimes e_{\alpha_1+\alpha_2}\big), \\
 & \big(e_{\alpha_2} \otimes f_{\alpha_2} ,\; -q^{2} f_{\alpha_2} \otimes e_{\alpha_2}  + \nu f_{\alpha_1+\alpha_2} \otimes e_{\alpha_1+\alpha_2}\big).
\end{align*}
Let $\ll$ denote the total ordering such that for every $\beta,\gamma \in \Delta^+$
we have $f_\beta \ll e_\gamma$ and 
$\beta\leq \gamma \in \Delta^+
$ 
implies that $ e_\beta \ll e_\gamma$ and $f_\gamma \ll f_\beta.$
Then $S_{\Delta^+}$ is a reduction system compatible with the ordering $\ll.$ Since $S_{\Delta^+}$ is composed of homogeneous polynomials of degree $2$ there are no inclusion ambiguities.  
It is easy to verify that the overlap ambiguities can be solved. As an example consider the expression $e_{\alpha_{1}} \otimes e_{\alpha_{2}}\otimes f_{\alpha_{1}}.$ We write the two sequences of reductions associated to $S_{\Delta^+}$ for this expression by enclosing between parenthesis the monomials on which a reduction acts non-trivially at each step. We have
\begin{align*}
    (e_{\alpha_{1}} \otimes e_{\alpha_{2}})\otimes f_{\alpha_{1}}& \mapsto -q^{-1}e_{\alpha_{2}}\otimes (e_{\alpha_{1}} \otimes f_{\alpha_{1}})
    \\&\mapsto q (e_{\alpha_{2}}\otimes  f_{\alpha_{1}}) \otimes e_{\alpha_{1}}+  q^{-1}\nu (e_{\alpha_{2}}\otimes f_{\alpha_{1}+\alpha_{2}}) \otimes e_{\alpha_{1}+\alpha_{2}}
    \\& \mapsto - f_{\alpha_{1}} \otimes  e_{\alpha_{2}} \otimes e_{\alpha_{1}} - f_{\alpha_{1}+\alpha_{2}} \otimes  e_{\alpha_2} \otimes e_{\alpha_{1}+\alpha_{2}}
\end{align*}
and 
\begin{align*}
    e_{\alpha_{1}} \otimes (e_{\alpha_{2}}\otimes f_{\alpha_{1}})& \mapsto -q^{-1}(e_{\alpha_{1}}\otimes f_{\alpha_{1}}) \otimes e_{\alpha_{2}}
    \\&\mapsto q^{-1}f_{\alpha_1}\otimes (e_{\alpha_1}\otimes e_{\alpha_2})+ q^{-1}\nu f_{\alpha_1+\alpha_2}\otimes (e_{\alpha_1+\alpha_2}\otimes e_{\alpha_2})
    \\& \mapsto - f_{\alpha_{1}} \otimes  e_{\alpha_{2}} \otimes e_{\alpha_{1}} - f_{\alpha_{1}+\alpha_{2}} \otimes  e_{\alpha_2} \otimes e_{\alpha_{1}+\alpha_{2}}.
\end{align*}

The other overlap ambiguities are solved in a similar way, hence from Bergmann's diamond lemma \cite{BergmanDiam} the set of algebra relations \ref{thm:THERELATIONS} is linearly independent and the spanning set given above is a basis.
\end{proof}

\begin{remark}
The form of the anti-holomorphic relations given above imply that the dc $\Omega^{(0,\bullet)}_q(\mathrm{F}_3)$ given in \cite[\textsection 5]{ROBPSLusz} is in fact the maximal prolongation of the fodc $\Omega^{(0,1)}_q(\mathrm{F}_3).$
\end{remark}

\begin{rem}
Unlike the fodc $\Omega^{(0,1)}_q(\mathrm{SU}_{n+1})$ considered in \cite[\textsection 3]{ROBPSLusz}, the maximal prolongation of $\Omega^1_q(\mathrm{SU}_3)$ does not have classical dimension, and restricts to a dc over $\OO_q(\mathrm{F}_{3})$ of non-classical dimension. This is our motivation for calculating the maximal prolongation of $\Omega^1_q(\mathrm{F}_3)$ directly.
\end{rem}

\subsection{Restriction to the Heckenberger--Kolb Calculus of the Quantum Projective Plane} \label{subsection:HKrestriction}

As explained in \textsection \ref{subsection:prelims.full.q.flag}, we have two copies of the quantum projective plane $\OO_q(\mathbb{CP}^2)$, arising as subalgebras of $\mathcal{O}_q(\mathrm{F}_3)$. Each subalgebra comes endowed with a left $\OO_q(\mathrm{SU}_3)$-covariant fodc of classical dimension $\Omega^1_q(\mathbb{CP}^2)$, known as the \emph{Heckenberger--Kolb fodc} \cite{HK}. For the copy of the quantum projective plane generated by the elements $z^{\alpha_1}_{ij}$, the quantum tangent space $T \subseteq \OO_q(\mathbb{CP}^2)^{\circ}$ of $\Omega^1_q(\mathbb{CP}^2)$ is given by  
\begin{align*}
T := \{F_1, \, F_2F_1, \, E_1, \, E_2E_1\}.
\end{align*}
The tangent space of the second copy of the quantum projective plane is given by 
\begin{align*}
T := \{F_2, \, F_1F_2, \, E_2, \, E_1E_2\}.
\end{align*}
The following lemma is an easy extension, for the rank $2$ case, of \cite[Proposition 4.2]{ROBPSLusz} to include the holomorphic parts of the Heckenberger--Kolb dc, and so, we omit it. 

\begin{lem}
The fodc $\Omega^1_q(\mathrm{F}_3)$ restricts to the Heckenberger--Kolb fodc for both copies of the quantum projective plane.
\end{lem}

Let $A$ be an algebra endowed with a dc $\Omega^{\bullet}(A)$, and let $B \subset A$ be a subalgebra. The \emph{restriction} $\Omega^{\bullet}(B)$ of $\Omega^{\bullet}(A)$ to $B$ is the subalgebra of $\Omega^{\bullet}(A)$ generated by the elements $b$ and $\exd b'$, for $b,\, b' \in B$. We note that $\Omega^{\bullet}(B)$ is of course a dc. We note that the first-order part  $\Omega^1(B)$ of the restriction dc is, of course, equal to the restriction of $\Omega^1(A)$ to $B$.

Denoting by $(\mathrm{Max}(\Omega^1(B)), \, \exd')$ the maximal prolongation of $\Omega^1(B)$ we have a surjection
\begin{align*}
p: \mathrm{Max}(\Omega^1(B)) \to \Omega^{\bullet}(B), 
\end{align*}
uniquely defined by 
$$
p(b_0\exd b_1 \wedge \cdots \wedge \exd b_k) = b_0\exd' b_1 \wedge \cdots \wedge \exd' b_k.
$$
It is clear that if $p$ is an injective map, then we have an isomorphism of dc.

For the case that $A$ is a Hopf algebra, $B$ a quantum homogeneous $A$-space, and $\Omega^{\bullet}(A)$ a left $A$-covariant dc, then it is clear that $p$ is a morphism in the category of relative Hopf modules. Thus we can see that $p$ is injective if and only if $\Phi(p)$ is injective.

For the case of the quantum projective plane, we can calculate the dimension of the restricted dc using the same approach as used in \cite[Proposition 4.8]{ROBPSLusz}. This shows us that the dimension is classical, just as is well-known for the Heckenberger--Kolb dc \cite[\textsection 3.3]{HKdR}. Thus we see that $\Phi(p)$ is injective, giving us the following proposition.

\begin{prop}
The dc $\Omega_q(\mathrm{F}_3)$ restricts to the Heckenberger--Kolb dc for both copies of $\OO_q(\mathbb{CP}^2)$.
\end{prop}

The question of how the covariant complex structures of the Heckenberger--Kolb dc relate to the possible covariant complex structures on $\Omega^1_q(\mathrm{F}_3)$ will be considered in \textsection \ref{section:complexstructures}.

\subsection{A Filtration for the Quantum Exterior Algebra}

Consider the following total order on the roots of $\frak{sl}_3$:
\begin{align}
\alpha_1 ~ \geq ~ -\alpha_1 ~ \geq ~ \alpha_2 ~ \geq ~ - \alpha_2 ~ \geq ~ \alpha_1 + \alpha_2 ~ \geq ~ - (\alpha_1 + \alpha_2). 
\end{align}
Following the approach of \cite[Appendix B]{ROBPSLusz}, this gives us a filtration on $V^{\bullet}$. We denote the  associated graded algebra by $\mathrm{gr}^{\mathscr{F}}$.

\begin{prop}
The algebra  $\mathrm{gr}^{\mathscr{F}}$ is generated by the elements $e_{\gamma}$, $f_{\gamma}$, for $\gamma \in \Delta^+$, subject to the relations, for all $\beta \leq \gamma \in \Delta^{+}$, 
\begin{align}
e_{\gamma} \wedge e_{\beta} = -q^{(\beta,\gamma)} e_{\beta} \wedge e_{\gamma}, & & f_{\gamma} \wedge f_{\beta} = -q^{(\beta,\gamma)} f_{\beta} \wedge f_{\gamma}, \\
f_{\gamma} \wedge e_{\beta} = -q^{(\beta,\gamma)} e_{\beta} \wedge f_{\gamma}.
\end{align}
\end{prop}
\begin{proof}
It is clear from the relation set for $V^{\bullet}$ given above that these relations hold in the associated graded algebra. Moreover, since these relations imply an obvious spanning of dimension $2^6$, we see that they must form a complete set of relations.    
\end{proof}

Following the argument of \cite[Proposition 3.18]{ROBPSLusz}, we can now prove the following corollary.

\begin{cor}
The algebra $V^{\bullet}$ is a Frobenius algebra.
\end{cor}

\begin{remark}
We note that since the space $V^{2^6}$ is a trivial $U_q(\frak{h})$-module $V^{\bullet}$ is a Frobenius algebra object in the category of $U_q(\frak{h})$-modules. This is in contrast to $V^{(0,\bullet)}$ which is a Frobenius algebra, but not in the category of $U_q(\frak{h})$-modules, nor are the anti-holomorphic subcomplexes of the Heckenberger--Kolb calculi.
\end{remark}

For a general Frobenius algebra $A$,  there exists an algebra automorphism $\sigma:A \to A$ of $A$, uniquely defined by the identity $B(x,y) = B(y,\sigma(x))$, for all $x,y \in A$. We see that the bilinear form $B$ of a Frobenius algebra is symmetric if and only if $\sigma = \id$.  With a view to describing the Nakayama automorphism of our quantum exterior algebra, we first describe the Nakayama automorphism of $\mathrm{gr}^{\mathscr{F}}$.

\begin{prop} 
The Nakayama automorphism of $\mathrm{gr}^{\mathscr{F}}$ is determined by 
\begin{align} \label{eqn:Nakayama}
\sigma(e_{\alpha_1}) = -q^{2} e_{\alpha_1}, & & \sigma(e_{\alpha_2}) = -q^{2} e_{\alpha_2}, & & \sigma(e_{\alpha_1 + \alpha_2}) = -q^{4} e_{e_{\alpha_1 + \alpha_2}}, \\
\sigma(f_{\alpha_1}) = -q^{-2} f_{\alpha_1}, & & \sigma(f_{\alpha_2}) = -q^{-2} f_{\alpha_2}, & & \sigma(f_{\alpha_1 + \alpha_2}) = -q^{-4} f_{\alpha_1 + \alpha_2}. \label{eqn:Nakayama2}
\end{align}
\end{prop}
\begin{proof}
In the following we let $\iota: \Lambda^n \rightarrow \Lambda^0$ denote the isomorphism of $\pi_B(A)-$comodules uniquely defined by 
\begin{align*}
    \iota(e_{\alpha_1} \wedge e_{\alpha_2} \wedge e_{\alpha_1+\alpha_2} \wedge f_{\alpha_1}\wedge f_{\alpha_2} \wedge f_{\alpha_1+\alpha_2})=1.
\end{align*}
Since $\sigma$ is an algebra map, it is clearly determined by its action of the algebra generators. For the generator $f_{\alpha_1}$ we have
\begin{align*}
(e_{\alpha_1} \wedge e_{\alpha_2} \wedge e_{\alpha_1+\alpha_2} \wedge f_{\alpha_2} \wedge f_{\alpha_1+\alpha_2}, f_{\alpha_1}) = & \, \iota[e_{\alpha_1} \wedge e_{\alpha_2} \wedge e_{\alpha_1+\alpha_2} \wedge f_{\alpha_2} \wedge f_{\alpha_1+\alpha_2} \wedge f_{\alpha_1}] \\
=& \, \iota[e_{\alpha_1} \wedge e_{\alpha_2} \wedge e_{\alpha_1+\alpha_2} \wedge f_{\alpha_1}\wedge f_{\alpha_2} \wedge f_{\alpha_1+\alpha_2} ] \\
=& \, 1.
\end{align*}
On the other hand, it holds that
\begin{align*}
(f_{\alpha_1}, e_{\alpha_1} \wedge e_{\alpha_2} \wedge e_{\alpha_1 + \alpha_2}\wedge f_{\alpha_2} \wedge f_{\alpha_1 + \alpha_2} ) =& \, \iota[f_{\alpha_1}\wedge e_{\alpha_1} \wedge e_{\alpha_2} \wedge e_{\alpha_1 + \alpha_2}\wedge f_{\alpha_2} \wedge f_{\alpha_1 + \alpha_2}]\\
=& \, -q^{-2} \iota[e_{\alpha_1} \wedge e_{\alpha_2} \wedge e_{\alpha_1+\alpha_2} \wedge f_{\alpha_1}\wedge f_{\alpha_2} \wedge f_{\alpha_1+\alpha_2} ] \\
=& \, -q^{-2}.
\end{align*}
Thus we see that $\sigma(f_{\alpha_1}) = q^{-2} f_{\alpha_1}$. The action of $\sigma$ on the other generators is calculated similarly. \end{proof}

In the following corollary we use our description of the Nakayama automorphism of $\mathrm{gr}^{\mathscr{F}}$ to produce an analogous description of the Nakayama automorphism of $\Lambda^{(0,\bullet)}_q$. The proof is completely analogous to the proof of \cite[Corollary 3.21]{ROBPSLusz}.

\begin{cor}  \label{cor:Nakayama}
The Nakayma automorphism of $\Lambda^{(0,\bullet)}_q$ acts on the generators $e_{\gamma}$, and $f_{\gamma}$, for $\gamma \in \Delta$, just as in \eqref{eqn:Nakayama} and \eqref{eqn:Nakayama2}.
\end{cor}

We finish by observing that  $\Lambda^{(0,\bullet)}_q$ is also a Koszul algebra. Recall that a Koszul algebra is a $\mathbb{Z}_{\geq 0}$-graded algebra admitting  a linear minimal graded free resolution. We refer the reader to the standard text \cite{Leonid.Quadratic} for more details on Koszul algebras.

\begin{prop}
The algebra $\Lambda^{(0,\bullet)}_q$ is a Koszul algebra.
\end{prop}
\begin{proof}
The algebra $\Lambda^{(0,\bullet)}_q$ is clearly a PBW-algebra in the sense of Priddy \cite[\textsection 4.1]{Leonid.Quadratic}. Thus it follows from  Priddy's theorem \cite[Theorem 3.1]{Leonid.Quadratic} that it is a Koszul algebra.
\end{proof}


\section{Connections and Torsion}

In this section we make some general observations about torsion for connections, and prove the existence of a covariant torsion-free connection for a dc over a quantum homogeneous space with cosemisimple quantum isotropy subgroup. Moreover, we classify covariant torsion free connections for such spaces. These general results are then applied to the dc $\Omega^{\bullet}_q(\mathrm{F}_3)$. One major difference with the irreducible quantum flag manifold case is that $\Omega^{\bullet}_q(\mathrm{F}_3)$ admits covariant connections with torsion.

Let us first recall the definition of a connection and its associated torsion operator. Let $\Omega^{\bullet}(B)$ be a differential calculus over an algebra $B$ and $\mathcal{F}$ a left $B$-module, a \emph{connection} on $\mathcal{F}$ is a $\mathbb{C}$-linear map
$
\nabla: \mathcal{F} \to \Omega^1(B) \otimes_B \mathcal{F}
$
satisfying the identity 
\begin{align*}
\nabla(bf) = \exd b \otimes f + b\nabla f, & & \textrm{ for all } b \in B, \, f \in \mathcal{F}.
\end{align*}
An immediate but important consequence of the definition is that the difference of two connections $\nabla - \nabla'$ is a left $B$-module map. 

Let $\nabla: \Omega^1(B) \to \Omega^1(B) \otimes_B \Omega^1(B)$ be a connection for $\Omega^1(B)$. The \emph{torsion} of $\nabla$ is the left $B$-module map 
\begin{align*}
T_{\nabla}:= \wedge \circ \nabla - \exd:\Omega^1(B) \to \Omega^2(B). 
\end{align*}
We note that if $B \subseteq A$ is a quantum homogeneous space, $\Omega^1(B)$ is a left $A$-covariant dc over $B$ in the category ${}^A_B\mathrm{Mod}$, and $\nabla$ is a left $A$-comodule map, then $T_{\nabla}$ is a morphism in ${}^A_B\mathrm{Mod}$.

We note that since the difference of two left $A$-covariant connections is a morphism of relative Hopf modules, the set of left $A$-covariant connections is an affine space for the vector space of morphisms from $\Omega^1(B)$ to $\Omega^1(B) \otimes_B \Omega^1(B)$. Moreover, for any two torsion free connections $\nabla$ and $\nabla'$, it holds that 
\begin{align*}
(\nabla - \nabla')(\omega) \in \ker(\wedge), & & \textrm{ for } \omega \in \Omega^1(A).
\end{align*}
This implies that the set of left $A$-covariant connections is an affine space for the vector space of morphisms from $\Omega^1(B)$ to the kernel
$$
\mathrm{ker}\Big(\wedge:\Omega^1(B) \otimes_B \Omega^1(B) \to \Omega^1(B)\Big).
$$

Moreover, we recall from \cite[\textsection 4.5]{HVBQFM} that if the space of $\pi_B(A)$-comodule maps from $V^{(0,1)}$ to $V^{(0,2)}$ is trivial, then this connection is necessarily torsion free. The following proposition is a variation on this result for the cosemisimple case.

\begin{prop}
Let $B \subseteq A$ be a quantum homogeneous space such that $\pi_B(A)$ is a cosemisimple Hopf algebra, and $A$ has no zero divisors. If $\Omega^{\bullet}(B)$ is a left $A$-covariant $*$-dc over $B$, then $\Omega^1(B)$ admits a left $A$-covariant torsion-free connection. 
\end{prop}
\begin{proof}
Let us first show that a left $A$-covariant connection always exits. Note first that since we are dealing with a quantum homogeneous space, and hence a principal comodule algebra, $\Omega^1_u(B)$ admits a left $A$-covariant connection $\widetilde{\nabla}$, see \cite[\textsection 3.4]{TBGS} for details. If we now assume that $A$ contains no zero divisors, then $\Omega^1(B)$ will be torsion-free as a left $B$-module, and since it is a $*$-fodc, it will also be torsion-free as a right $B$-moodule. Since $\Omega^1(B)$ is also projective as a left $B$-module, this means that we can quotient $\widetilde{\nabla}$ to get a left $A$-covariant connection $\nabla$ for any left $A$-covariant fodc $\Omega^1(B)$.

If $\nabla$ is torsion-free, then we are done. So let us assume that $\nabla$ has non-zero torsion $\mathrm{Tor}(\nabla)$. Since we are assuming that $\pi_B(A)$ is a cosemisimple Hopf algebra, we can choose a left $\pi_B(A)$-comodule splitting of the surjection $\wedge:V^1 \otimes V^1 \to V^2$, and hence a splitting $i$, in the category ${}^A_B\mathrm{Mod}$, of the surjection 
$$
\wedge:\Omega^1(B) \otimes_B \Omega^1(B) \to \Omega^1(B)
$$
Consider next the left $A$-comodule map 
$$
\nabla' := \nabla - i \circ T_\nabla: \Omega^1(B) \to \Omega^1(B) \otimes_B \Omega^1(B).
$$
Now, for $\omega \in \Omega^1(B)$, and $b \in B$, we have 
\begin{align*}
\nabla'(b\omega) = & \, \nabla(b\omega) - i \circ T_\nabla(b \omega)\\
= & \, \exd b \otimes \omega + b \nabla(\omega) - b(i \circ T_\nabla (\omega))\\
= & \, \exd b \otimes \omega + b\big(\nabla(\omega) - i \circ T_\nabla\big)(\omega)\\
= & \, \exd b \otimes \omega + b \nabla'(\omega).
\end{align*}
Thus we see that $\nabla'$ is a connection. Next we note that 
\begin{align*}
T_\nabla' & \, = \wedge \circ \nabla' - \exd \\
& \, =  \wedge \circ (\nabla - \iota \circ T_\nabla) - \exd\\
& \, =  \wedge \circ \nabla - \exd - T_\nabla\\
    & \, =  T_\nabla - T_\nabla\\
    & \, =   0.
\end{align*}
Thus we see that $\nabla'$ is torsion-free, and hence that a left $A$-covariant torsion-free connection always exists.
\end{proof}

A simple but useful observation is that if $\Omega^1(B)$ admits a unique left $A$-covariant connection, then it must be the same as the torsion-free connection just constructed. This gives us the following corollary.

\begin{cor} \label{cor:unique.torsion}
Let $B \subseteq A$ and $\Omega^{\bullet}(B)$ be as above. If $\Omega^1(B)$ admits a unique left $A$-covariant connection, then this connection is necessarily torsion-free.
\end{cor}

\begin{eg}
As an application of the above corollary, consider the Heckenberger--Kolb differential calculi for the irreducible quantum flag manifolds, a special subclass of the quantum flag manifolds, itself a family of Drinfeld--Jimbo quantum homogeneous spaces, containing $\OO_q(\mathrm{F}_3)$ and $\OO_q(\mathbb{CP}^2)$. These are covariant differential calculi, extending the Podle\'s calculus discussed in the introduction. In \cite{HolVBs} their $2$-forms were shown to possess a unique covariant connection, and moreover, this connection was shown to be torsion-free using a representation theoretic argument. We now see that vanishing of the torsion follows directly from Corollary \ref{cor:unique.torsion}.
\end{eg}


We now apply these general results to the special case of  $\Omega^1(\mathrm{F}_3)$. First we calculate the dimension of the affine space of torsion free connections. In particular, we see that in general, torsion free connections are not unique. 

\begin{cor}
For the full quantum flag manifold $\mathcal{O}_q(\mathrm{F}_3)$, endowed with the dc $\Omega^{\bullet}_q(\mathrm{F}_3)$, the affine space of connections has dimension $12$. The dimension of the affine space of torsion-free connections is $6$. 
\end{cor}
\begin{proof}
The dimension of the space of $U_q(\frak{h})$-module maps from $V^1$ to $V^1 \otimes V^1$ can be calculated by noting the multiplicity of the weight spaces of $V^1 \otimes V^1$ of weight $\pm \alpha_1, \pm \alpha_2,$ and $\pm (\alpha_1+\alpha_2)$. Looking at Table \ref{table:sumofroots} in Appendix \ref{app:sl3}, we see that each weight has multiplicity $2$. Hence the dimension of the affine space of connections is $12$. With a similar argument one can confirm that the dimension of the affine space of torsion-free connections is $6$.
\end{proof}


\section{Almost-Complex Structures for the Lusztig--de Rham Complex} \label{section:complexstructures}

In this section we examine covariant complex and almost complex structures for the dc $\Omega^1_q(\mathrm{F}_3)$. We observe that the number of almost-complex structures decreases from $8$ (which is  $2$ to the number of positive roots of $\frak{sl}_3$) to $4$ (which is $2$ to the number of simple roots of $\frak{sl}_3$). Furthermore, we demonstrate that all of these almost-complex structures are integrable, which is to say, they are both complex structures.

\subsection{Preliminaries on Complex and Almost-Complex Structures}

In this subsection, we briefly recall some preliminaries about almost-complex and complex structures. See \cite[\textsection 1]{BeggsMajid:Leabh} or \cite{BS,KLvSPodles,MMF2} for a more detailed discussion of complex structures.

An {\em almost complex structure} $\Om^{(\bullet,\bullet)}$, for a  $*$-dc  $(\Om^{\bullet},\exd)$, is an $\mathbb{Z}^2_{\geq 0}$-algebra grading 
of $\Om^{\bullet}$ such that
\begin{align*} 
\Om^k = \bigoplus_{a+b = k} \Om^{(a,b)}, & & \big(\Om^{(a,b)}\big)^* = \Om^{(b,a)}, & & \textrm{ for all } (a,b) \in \mathbb{Z}^2_{\geq 0}. 
\end{align*}
If the following multiplication maps
\begin{align*}
\wedge: \Omega^{(a,0)} \otimes_B \Omega^{(0,b)} \to \Omega^{(a,b)}, & & \wedge: \Omega^{(0,b)} \otimes_B \Omega^{(a,0)} \to \Omega^{(a,b)},
\end{align*}
are isomorphisms, for all $(a,b) \in \mathbb{Z}^2_{\geq 0}$, then we say that the almost-complex structure is \emph{factorisable}.

If the exterior derivative decomposes into a sum $\exd = \del + \adel$, for $\del$ a (necessarily unique) degree $(1,0)$-map, and $\adel$ a (necessarily unique) degree $(0,1)$-map, then we say that $\Om^{(\bullet,\bullet)}$ is a \emph{complex structure}. It follows that we have a double complex. The {\em opposite} complex structure of a complex structure $\Om^{(\bullet,\bullet)}$ is the \mbox{$\mathbb{Z}^2_{\geq 0}$}-\alg grading  $\overline{\Om}^{(\bullet,\bullet)}$, defined by $\ol{\Om}^{(a,b)} := \Om^{(b,a)}$, for $(a,b) \in \mathbb{Z}_{\geq 0}^2$. 

 Finally, we restrict to the case of a covariant dc $\Omega^{\bullet}$ over a quantum homogeneous space $B \subseteq A$. In this case, a complex structure $\Omega^{(\bullet,\bullet)}$ for $\Omega^{\bullet}$ is said to be \emph{covariant} if the $\mathbb{Z}^2_0$-decomposition of the dc is a decomposition in the category of two-sided relative Hopf modules ${}^A_B\mathrm{Mod}_B$.

\subsection{Almost-Complex Structures for the Classical Full Flag Manifold $\mathrm{F}_3$}

In this subsection we briefly recall the covariant almost complex structures for the classical flag manifold $\mathrm{F}_3$. We do so to highlight the novel non-classical behavior occurring for the quantum case. We refer the interested reader to \cite{Alex.Flag.Yugo} or \cite{BastonEastwood} for more further details.

A choice of almost-complex structure for the manifold $F_3$ corresponds to assigning to each positive root of the root system $\Delta$ of $\frak{sl}_3$ the label of holomorphic or anti-holomorphic. We see that there exist eight such labellings, meaning that up to identification of opposite almost-complex structures, we have four. 

A labeling corresponds to a complex structure if and only if it gives a choice of positive roots for $\Delta$, or equivalently a choice of base for the root system. We see that three of our almost-complex structures are integrable and one is not. The Weyl group $S_3$ of $\frak{sl}_3$ acts transitively on the set of bases for $\Delta$, and hence on the set of covariant almost-complex structures. We collect these recollections below in the form of a table.


\subsection{First-Order Almost-Complex Structures }

As usual in the theory of differential calculi, we find it convenient to initially work at the level of fodc and then discuss the extension to higher forms. This motivates the following general definition. 
\begin{table}
\label{tab.complex.structures}
  \centering
  \begin{tabular}{|c|c|c|}
    \hline
      &          &  \\
Roots & $T^{(0,1)}$ & Integrable  \\
  &          &  \\
    \hline
  &          &  \\
$\alpha_1, \, \alpha_2, \, \alpha_1+\alpha_2$   &  $\Big\{E_1,\, E_2, \, [E_1,E_2] \Big\}$ & \ding{51}  \\
  &  &  \\
$-\alpha_1, \, \alpha_2, \, \alpha_1+\alpha_2$   &  $\Big\{F_1,\, E_2, \, [E_1,E_2] \Big\}$,  & \ding{51}  \\
  &      &  \\
$\alpha_1, \, - \alpha_2, \, \alpha_1+\alpha_2$   &  $\Big\{E_1,\, F_2, \, [E_1,E_2] \Big\}$ & \ding{51}   \\
  &      &  \\
$\alpha_1, \, \alpha_2, \, -\alpha_1-\alpha_2$   &  $\Big\{E_1,\, E_2, \, [F_1,F_2] \Big\}$ & \ding{55} \\
  &          &  \\
    \hline
  \end{tabular}
  \label{tab:complex.structures}
\end{table}
\begin{defn}
A \emph{first-order almost complex structure}, or \emph{foacs}, for a $*$-fodc $\Omega^1(B)$ over an algebra $B$ is a direct sum decomposition of $B$-bimodules
\begin{align} \label{eqn:foacs}
\Omega^1(B) \cong \Omega^{(1,0)} \oplus \Omega^{(0,1)}
\end{align}
such that $(\Omega^{(1,0)})^* = \Omega^{(0,1)}$ or equivalently  $(\Omega^{(0,1)})^* = \Omega^{(1,0)}$.
\end{defn}

Just as for a complex structure, we have the corresponding notions of \emph{opposite foacs} and \emph{covariant foacs}  for a covariant fodc over a quantum homogeneous space. Moreover, we note that any covariant dc over a quantum homogeneous space $B \subseteq A$, a covariant foacs implies a corresponding decomposition of the cotangent space $V^1$, in the category ${}^{\pi_B}\mathrm{Mod}_B$  of the dc over $B$.

Let us now recall a formula detailing the interaction of the dc $*$-map of a dc over a Hopf algebra $A$ with the fundamental theorem of Hopf modules. Consider the commutative diagram
\begin{align}
\xymatrix{ 
\Omega^1(A)   \ar[rrr]^{\unit~}                   & & &   A \otimes \Lambda^1  \\
\Omega^1(A)  \,\,   \ar[u]^{*}                    & & &   A \otimes \Lambda^1. \ar[lll]^{\unit^{-1}} \ar[u]_{~ \unit\, \circ \, * \, \circ \, \unit^{-1}}
}
\end{align}
As is easily shown (see \cite[\textsection 2.6]{MMF2}) the map $\unit \circ * \circ \unit^{-1}$ acts explicitly as
\begin{align} \label{eqn:localstar}
\unit \circ * \circ \unit^{-1}(a \otimes v)= - a_{(2)}^* \otimes v^*a_{(1)}^*, & & \textrm{ for } a \otimes v \in A \otimes \Lambda^1,
\end{align}
where the star map $\ast:\Lambda^1 \to \Lambda^1$ is defined by $[a]^* = [S(a)^*]$.


\begin{lem} \label{lem:foacs}
Let $B \subseteq A$ be a quantum homogeneous space, and $\Omega^1(A)$ a left $A$-covariant dc for $A$ that frames $\Omega^1(B)$, the restriction of $\Omega^1(A)$ to a fodc on $B$. Consider a decomposition of $V^1$, the cotangent space of $\Omega^1(B)$,
\begin{align} \label{eqn:Vdecomp}
V^1 \cong V^{(1,0)} \oplus V^{(0,1)} \in {}^{\pi_B}\mathrm{Mod}_B
\end{align}
that is moreover, a decomposition of right $A$-modules, with respect to the embedding of $V^1$ in $\Lambda^1$, the tangent space of $\Omega^1(A)$. Then the corresponding decomposition $\Omega^1(B) \cong \Omega^{(1,0)} \oplus \Omega^{(0,1)}$ is a left $A$-covariant foacs if $V^{(1,0)}$ and $V^{(0,1)}$ are interchanged by the $*$-map of $\Lambda^1$. 
\end{lem}
\begin{proof}
Consider the $A$-subbimodule of $\Omega^1(A)$ given by 
$$
A\Omega^1(B) \cong A \otimes V^1 \cong  A\Omega^1(B)A,
$$
where the second isomorphism follows from the fact that $V^1$ is a right $A$-module, since we have assumed that $\Omega^1(A)$ is a framing calculus for $\Omega^1(B)$. The decomposition of $V^1$ gives us the decomposition
$$
A\Omega^1(B) \cong  A\Omega^{(1,0)}(B) \oplus A\Omega^{(0,1)}(B) \cong (A \otimes V^{(1,0)}) \oplus (A \otimes V^{(0,1)}).
$$
This is again a decomposition of $A$-bimodules, since by assumption the decomposition of $V^1$ is a decomposition of right $A$-modules. Since we are supposing that $\ast$ maps $V^{(1,0)}$ to $V^{(0,1)}$, and that both subspaces are right $A$-modules, it follows from \eqref{eqn:localstar} that the $*$ map interchanges $A \otimes V^{(1,0)}$ and $A \otimes V^{(0,1)}$. Since $\Omega^1(B)$ is a $*$-fodc, this of course implies that $*$ interchanges $\Omega^{(1,0)}$ and $\Omega^{(0,1)}$, and so, we have a foacs.
\end{proof}

\subsection{Complex Structures for Full Quantum Flag Manifolds}

In this subsection we classify the covariant complex structures on the dc $\Omega^1_q(\mathrm{F}_3)$. We find that two of the classical almost complex structures fail to extend to the quantum setting. In particular, one of the bases of the root system of $\frak{sl}_3$ fails to have a corresponding foacs in the quantum setting. This breaks the classical Weyl group symmetry of the almost-complex structures on $\mathrm{F}_3$.

\begin{prop} \label{prop:almost.complex.class}
The fodc $\Omega^1_q(\mathrm{F}_3)$ admits, up to identification of opposite structures, two covariant foacs. Explicitly, one decomposition of $V^1$ is given by 
\begin{align*}
    V^{(1,0)} = \mathrm{span}_{\mathbb{C}}\Big\{e_{\alpha_1},\, e_{\alpha_2}, \, e_{\alpha_1+\alpha_2}\Big\}, & & V^{(0,1)} := \mathrm{span}_{\mathbb{C}}\Big\{ f_{\alpha_1},\, f_{\alpha_2}, \, f_{\alpha_1+\alpha_2}\Big\},
\end{align*}
and the other is given by 
\begin{align*}
    V^{(1,0)} = \mathrm{span}_{\mathbb{C}}\Big\{f_{\alpha_1},\, e_{\alpha_2}, \, e_{\alpha_1+\alpha_2}\Big\}, & & V^{(0,1)} := \mathrm{span}_{\mathbb{C}}\Big\{ e_{\alpha_1},\, f_{\alpha_2}, \, f_{\alpha_1+\alpha_2}\Big\}.
\end{align*}
\end{prop}
\begin{proof}
Consider a general left $\OO_q(\mathrm{SU}_3)$-covariant foacs on $\Omega^1_q(\mathrm{F}_3)$, and denote by 
$$
V^1 \cong V^{(1,0)} \oplus V^{(0,1)}.
$$
the corresponding decomposition of the cotangent space $V^1$ into two left $\mathcal{O}(\mathbb{T}^2)$-comodule right $\OO_q(\mathrm{F}_3)$-modules. Since the basis elements all have mutually distinct weights, we see that each basis element is contained in either $V^{(1,0)}$ or $V^{(0,1)}$.
The right $\OO_q(\mathrm{F}_3)$-module requirement, together with Lemma \ref{lem:uijeij}, implies that if $e_{\alpha_1}$ is contained in $V^{(1,0)}$, then $e_{\alpha_1+\alpha_2}$ is also contained in $V^{(1,0)}$, and analogously, if $f_{\alpha_1}$ is contained in $V^{(0,1)}$, then $f_{\alpha_+\alpha_2}$ is contained in $V^{(0,1)}$. In other words, any complex structure is determined by knowing whether the basis elements $e_{\alpha},f_{\alpha}$, for $\alpha \in \Pi$, are contained in $V^{(1,0)}$ or $V^{(0,1)}$.

We now note that any such $\OO_q(\mathrm{F}_3)$-decomposition of $V^1$ will necessarily be a decomposition of right $\OO_q(\mathrm{F}_3)$-modules. This allows us to appeal to Lemma \ref{lem:foacs}. Considering $V^1$ as a subspace of $\Lambda^1$, the cotangent space of the fodc $\Omega^1_q(\mathrm{SU}_3)$, and recalling that $e_{\gamma}^* = f_{\gamma}$, for all $\gamma \in \Delta^+$, we now see that the only possible decompositions are those two decompositions given in the statement of the proposition.
\end{proof}

Given a foacs on a fodc, there is at most one extension to an almost complex structure on its maximal prolongation, or indeed any quotient thereof (see \cite[Proposition 6.1]{MMF2} for details). The following proposition tells that both our foacs extend. 

\begin{cor}
Both foacs on $\Omega^1_q(\mathrm{F}_3)$ extend to a factorisable almost complex structure on $\Omega^{\bullet}_q(\mathrm{F}_3)$.
\end{cor}
\begin{proof}
The fact that both first-order structures extend to covariant almost-complex structures, follows directly from the explicit form of the relations given in Theorem \ref{thm:THERELATIONS} and \cite[Theorem 6.4]{MMF2}. Moreover, factorisability follows from the explicit form of the relations and \cite[Corollarly 6.8]{MMF2}.
\end{proof}

\subsection{Integrability for the Full Quantum Flag Almost-Complex Structures}

As shown in \cite[Lemma 7.2]{MMF2}, an almost-complex structure $\Omega^{(\bullet,\bullet)}$ on a dc $\Omega^{\bullet}$ is integrable if and only if the maximal prolongation of the fodc $\Omega^{(0,1)}$ is isomorphic to the subalgebra $\Omega^{(0,\bullet)}$. Using this reformulation of integrability, we now observe that, just as in the classical case, both the covariant almost-complex structures on $\Omega^{\bullet}_q(\mathrm{F}_3)$ are integrable. Interestingly, this means that $\Omega^{\bullet}_q(\mathrm{F}_3)$ does not admit a non-integrable covariant almost-complex structure.

\begin{prop}
Both covariant almost-complex structures of the dc $\Omega^{\bullet}_q(\mathrm{F}_3)$ are integrable.
\end{prop}
\begin{proof}
We will treat the case of the almost-complex structure 
$$
V^{(0,1)} = \big \{ e_{\alpha_1}, \, e_{\alpha_2}, \, e_{\alpha_1+\alpha_2} \big\},
$$
the other case being entirely analogous. We need to calculate the dimension of the maximal prolongation of the associated fodc $\Omega^{(0,1)}$. We note that $\Omega^{(0,1)}_q(\mathrm{SU}_3)$ is a framing calculus for $\Omega^{(0,1)}_q(\mathrm{F}_3)$, allowing us to use the appraoch of \textsection \ref{subsection:remarksQHTS} to calculate the degree two relations of the maximal prolongation of $\Omega^{(0,1)}_q(\mathrm{F}_3)$.

We see that the ideal $I' \subseteq \mathcal{O}_q(\mathrm{F}_3)^+$ corresponding  to the $\Omega^{(0,1)}_q(\mathrm{F}_3)$ contains the elements 
$$
I \cup \{z^{\alpha_1}_{12},\, z^{\alpha_2}_{23}, \, z^{\alpha_1}_{13}\}.
$$
Moreover, since the quotient of $\mathcal{O}_q(\mathrm{F}_3)^+$ by $I'$ is three-dimensional, we see that this is in fact the whole ideal. 

Operating on the elements of $I$ by $\omega$ we clearly reproduce the degree-$(0,2)$ elements from those given in Theorem \ref{thm:thereal}. For the element $z^{\alpha_2}_{23}$, recalling \cite[Lemma 3.8]{ROBPSLusz} we see that 
\begin{align*}
\omega(z^{\alpha_2}_{23}) = \omega(u_{23}S(u_{33})) = & \, \sum_{a=1}^3 [u_{23}S(u_{b3})] \otimes [S(u_{3b})^+] + \sum_{a=1}^3 [S(u_{b3})^+] \otimes [u_{23}^+S(u_{3b})] \\
 & ~~~~~ + \sum_{a=1}^3 [u_{2a}^+S(u_{b3})] \otimes [u_{a3}^+S(u_{3b})].
\end{align*}
Since each of the elements 
$$
u_{23}, \, u_{13}, \, u_{22}^+, \, u_{23}
$$
pair trivially with each element of $T^{(0,1)}$, we now see that $\omega(z^{\alpha_2}_{23}) = 0$. Analogous calculations establish that 
$$
\omega(z^{\alpha_2}_{23}) = \omega(z^{\alpha_1}_{13}) = 0.
$$
Thus we see that the maximal prolongation of $\Omega^{(0,1)}_q(\mathrm{F}_3)$ is isomorphic to the subalgebra $\Omega^{(0,\bullet)}_q(\mathrm{F}_3)$, and so, the almost-complex structure is integrable.
\end{proof}


\subsection{Restriction of the Almost Complex Structures}

Throughout this subsection, $P$ will denote a $*$-algebra and $B$ a $*$-subalgebra. We note that, for $\Omega^{\bullet}(P)$ a $*$-dc over $P$, the restriction to a dc on $B$ is again a $*$-dc calculus. 

\begin{prop}
Let $\Omega^{\bullet}(P)$ be a $*$-dc over $P$, and let $\Omega^{(\bullet,\bullet)}(P)$ be an almost complex structure for $\Omega^{\bullet}(P)$. Denote by $\Omega^{\bullet}(B)$  the restriction of $\Omega^{\bullet}(P)$ to a $*$-dc on $B$. Then an almost complex structure on $\Omega^{\bullet}(B)$ is given by $\Omega^{(\bullet,\bullet)}(B)$, where 
\begin{align*}
\Omega^{(a,b)}(B) := \Omega^{(a,b)}(P) \cap \Omega^{a+b}(B)
\end{align*} 
if and only if the following three equivalent conditions hold
\begin{enumerate}
\item \label{cond:1}  $\del b \in \Omega^1(B)$,
\item \label{cond:2}  $\adel b \in \Omega^1(B)$,
\item \label{cond:3} $\Omega^1(B)$ is homogeneous with respect to the decomposition $\Omega^1(P) \cong \Omega^{(1,0)} \oplus \Omega^{(0,1)}$. 
\end{enumerate}
\end{prop}
\begin{proof}
Let us first show that \eqref{cond:3} is equivalent to having an almost complex manifold. Since $\Omega^{\bullet}(B)$ is a $*$-subspace of $\Omega^{\bullet}(P)$, we see that, for all $(a,b) \in \mathbb{Z}^2_{\geq 0}$,
\begin{align*}
(\Omega^{(a,b)}(B))^* = & \, (\Omega^{(a,b)}(P) \cap (\Omega^{a+b}(B)))^* = \Omega^{(b,a)}(P) \cap \Omega^{a+b}(B) \subseteq \Omega^{(b,a)}(B). 
\end{align*}
Thus we see that $(\Omega^{(a,b)}(B))^* = \Omega^{(b,a)}(B)$. It remains to show that homogeneity of $\Omega^{\bullet}(B)$ with respect to the $\mathbb{Z}^2_{\geq 0}$-grading $\Omega^{(\bullet,\bullet)}(B)$ is equivalent to homogeneity of $\Omega^{1}(B)$ with respect to the grading. One direction is obvious, so let us assume homogeneity of $\Omega^1(B)$ with respect to the decomposition $\Omega^1(P) \cong \Omega^{(1,0)}(P) \oplus \Omega^{(0,1)}(P)$. Every form in $\Omega^k(B)$ is a linear combination of elements of the form 
$$
b_0\exd b_1 \wedge \cdots \wedge \exd b_k = b_0(\del b_1 + \adel b_1) \wedge \cdots \wedge (\del b_k + \adel b_k).
$$
Hence, each $\omega \in \Omega^k(B)$ is a $B$-linear combination of products of $1$-forms of the form $\del b$ or $\adel c$, for $b,c \in B$. Since each product is homogeneous with respect to the $\mathbb{Z}^2_{\geq 0}$-grading, and each $\Omega^{(a,b)}(B)$ is a left $B$-module. Moreover, since $\del b$ and $\adel c$ are in $\Omega^1(B)$ by assumption, these products are actually contained in $\Omega^{\bullet}(B)$. Thus, we see that, as required, $\Omega^{\bullet}(B)$ is a homogeneous subspace with respect to the $\mathbb{Z}^2_{\geq 0}$-grading.

It remains to show that \eqref{cond:1}, \eqref{cond:2}, and \eqref{cond:3} are equivalent. Note first that since $\Omega^{(\bullet,\bullet)}(P)$ is an almost-complex structure, $\del b$ is contained in $\Omega^1(B)$ if and only if $\adel b^*$ is contained in $\Omega^1(B)$. Thus \eqref{cond:1} and \eqref{cond:2} are indeed equivalent. However, \eqref{cond:1} and \eqref{cond:2} are together equivalent to \eqref{cond:3}, and so, we are done. 
\end{proof}

The proof of the following corollary, discussing the relationship of integrability and restriction, is clear, and so, we omit it. 

\begin{cor}
If $\Omega^{(\bullet,\bullet)}(P)$ is an integrable complex structure that restricts to an almost complex structure $\Omega^{(\bullet,\bullet)}(B)$ on $B$, then $\Omega^{(\bullet,\bullet)}(P)$ is also integrable.
\end{cor}


\subsection{Restriction of Covariant Almost Complex Structures}

In this subsection we deal with the restriction of covariant almost-complex structures for nested of pairs of quantum homogeneous spaces. Throughout $A$ will denote a Hopf algebra, and $P \subseteq A$ and $B \subseteq A$ a pair of quantum homogeneous $A$-spaces, such that $B \subseteq P$, that is to say a \emph{nested pair} of quantum homogeneous spaces \cite{GAPP}. 
Moreover, let $\Omega^{\bullet}(P)$ be a left $A$-covariant dc over $P$ and $\Omega^{\bullet}(B)$ the restriction to a dc over $B$. Since we have two quantum homogeneous spaces, we have two versions of Takeuchi's equivalence. We denote the functors of the two equivalences by $\Phi_P$ and $\Psi_P$ for $P$, and by $\Phi_P$ and $\Psi_P$ for $B$. Moreover, we denote $V^1_P:= \Phi_P(\Omega^1(P))$ and $V^1_B:= \Phi_B(\Omega^1(B))$.

\begin{prop} \label{prop:ALC.Restriction.Cotangent}
Assume that the embedding 
\begin{align*}
\iota: V_B \mapsto V_P, & & [\exd b] \mapsto [\exd b]
\end{align*}
is injective, and identify $V^1_B$ with its image. Then any left $A$-covariant almost-complex structure on $\Omega^{\bullet}(P)$ descends to a complex structure on $\Omega^{\bullet}(B)$ if and only if one, or equivalently both, of the subspaces 
\begin{align*}
V^{(1,0)}_B := V^1_B \cap V^{(1,0)}_P, & & V^{(0,1)}_B := V^1_B \cap V^{(0,1)}_P,
\end{align*}
are $\pi_B(A)$-subcomodules of $V^1_B$, and $V^1_B \cong V^{(1,0)}_B \oplus V^{(0,1)}_B$.
\end{prop}
\begin{proof}
Let us begin by assuming that $V^1_B$ is homogeneous with respect to the decomposition of $V^1_P$. Denoting
$$
\exd b =: \sum_i a_i \otimes v_i \in A \square_{\pi_B} V^1_B,
$$
homogeneity of $V^1_B$ gives us a decomposition of each $v_i$ by $v_i = v^+_i + v^-_i$, where $v^+_i \in V^{(1,0)}_B$ and $v^-_i \in V^{(0,1)}_B$. This gives us the identity
$$
\del b = \sum_i a_i \otimes v_i^+ \in A \otimes V^{(1,0)}.
$$
Now if $V^{(1,0)}_B$ and $V^{(0,1)}_B$ are left $\pi_B(A)$-comodules, then the fact that $\exd b$ is an element of $A \square_{\pi_B} V^1_B$, together with the definition of the coproduct, implies that $\del b$ is an element of $A \square_{\pi_B} V^1_B$, which is to say, $\del b$ is an element of $\Omega^1(B)$, which is to say the complex structure is an almost complex structure. 

In the other direction, assume that the almost complex structure on $\Omega^1(P)$ restricts to a complex structure on $\Omega^1(B)$. In particular, assume that $\adel b$ is an element of $\Omega^1(B)$. Then since $[\exd b] = [\del b] + [\adel b]$, and $[\del b], [\adel b] \in \iota(V^1_B)$, we see that  $V^1_B = V^{(1,0)}_B \oplus V^{(0,1)}_B$. It follows from Takeuchi's equivalence that this is a decomposition in the category ${}^{\pi_B}\mathrm{Mod}_B$, and in particular a decomposition of left $\pi_B(A)$-comodules. Thus we have established the opposite implication.
\end{proof}

The following corollary is a simple dualisation of this result to the tangent space setting, under the assumption that $B$ is a quantum homogeneous space of the form ${}^WA$, for $W \subseteq A^{\circ}$, as discussed in \textsection \ref{subsection:remarksQHTS}. Note that a covariant almost complex structure $\Omega^{(\bullet,\bullet)}(P)$ on $\Omega^{\bullet}(P)$ induces a direct sum decomposition of its corresponding the tangent space $T \cong T^{(1,0)} \oplus T^{(0,1)}$, where $T^{(1,0)}$ is the subspace of elements of $T$ that vanish on $V^{(1,0)}$, and $T^{(0,1)}$ is the subspace of elements of $T$ that vanish on $V^{(0,1)}$.

\begin{cor} \label{cor:tangent.AL.restriction}
The almost complex structure $\Omega^{(\bullet,\bullet)}$ restricts to an almost complex structure on $\Omega^{\bullet}(B)$ if 
\begin{align}\label{eqn:tangent.Levi.AC}
WT^{(1,0)}|_B = T^{(1,0)}|_B, & & \textrm{ and }  & & WT^{(0,1)}|_B = T^{(0,1)}|_B.
\end{align}
\end{cor}
\begin{proof}
Note first that if \eqref{eqn:tangent.Levi.AC} holds then $WT|_B = T|_B$, and hence the map $\iota$ is an injection (see \cite[\textsection 4.2]{ROBPSLusz} for a discussion of this). The equivalence of the requirements of \eqref{eqn:tangent.Levi.AC} and those given in Proposition \ref{prop:ALC.Restriction.Cotangent} now follows from a routine dualisation argument.
\end{proof}


\subsection{Restriction of the Complex Structures to $\OO_q(\mathbb{CP}^2)$}

In this subsection we address the question of the restriction of the complex structures on $\Omega^{\bullet}_q(\mathrm{F}_3)$ to the dc $\Omega^{\bullet}_q(\mathbb{CP}^2)$. We see that just as in the classical case, sometimes a complex structure restricts, while other times it does not. In particular, we see that the unique left $\OO_q(\mathrm{SU}_3)$-covariant complex structures on the Heckenberger--Kolb complex dc can be realised as the restriction of a complex structure on $\Omega^{\bullet}_q(\mathrm{F}_3)$.

\begin{prop}
It holds that 
\begin{enumerate}
    \item the complex structures $V^{(\bullet,\bullet)}_I$ and $V^{(\bullet,\bullet)}_{II}$ restrict to a left $\OO_q(\mathrm{SU}_3)$-covariant complex structure on $\Omega^{\bullet}_q(\mathbb{CP}^2_{\alpha_1})$,
    \item the complex structure $V^{(\bullet,\bullet)}_{I}$ restricts to a a left $\OO_q(\mathrm{SU}_3)$-covariant complex structure on $\Omega^{\bullet}_q(\mathbb{CP}^2_{\alpha_2})$, while $V^{(\bullet,\bullet)}_{II}$ does not restrict.
\end{enumerate}
\end{prop}
\begin{proof}
By Corollary \ref{cor:tangent.AL.restriction} we simply need to check that $T^{(1,0)}$ and $T^{(0,1)}$ are $U_q(\frak{l}_S)$-modules. For example, for the complex structure $V^{(\bullet,\bullet)}_{II}$, 
\begin{align*}
F_1 \triangleright F_2 = F_1F_2 \notin T^{(0,1)},
\end{align*}
and so, the complex structure does not restrict to a complex structure on $\Omega^{\bullet}_q(\mathbb{CP}^2_{\alpha_2})$.
\end{proof}

\begin{remark}
The asymmetry between the case of $\Omega^{\bullet}_q(\mathbb{CP}^2_{\alpha_1})$ and $\Omega^{\bullet}_q(\mathbb{CP}^2_{\alpha_2})$ can be understood as follows: In the classical case, the complex structure on each copy of the complex projective plane can be realised as the restriction of two distinct complex structures on $\mathrm{F}_3$. However, in the noncommutative setting, we have fewer complex structures, meaning we have only one \emph{lift} of the complex structure on $\Omega^{\bullet}_q(\mathbb{CP}^2_{\alpha_2})$ to a complex structure on $\Omega^{\bullet}_q(\mathrm{F}_3)$. However, if we instead look at the case of the Lusztig dc on $\OO_q(\mathrm{SU}_3)$ associated to the reduced decomposition of the longest element of the Weyl group $w_0 = s_1s_1s_1$, then this situation is reversed, with $\Omega^{(\bullet,\bullet)}_q(\mathbb{CP}^2_{\alpha_2})$ having two lifts and $\Omega^{(\bullet,\bullet)}_q(\mathbb{CP}^2_{\alpha_1})$ having only one. Thus the symmetry is preserved by considering the alternative reduced decomposition. 
\end{remark}


\subsection{Some Remarks about the Higher Rank Full Quantum Flag Manifolds}

In this subsection, which is in effect an extended remark, we discuss the extension of our results for $\OO_q(\mathrm{F}_3)$ to the higher rank quantum flag manifolds. The definition of the quantum flag manifolds directly extends from the Podle\'s sphere, and $\OO_q(\mathrm{F}_3)$, to a general definition of full quantum manifold. Following the conventions of \cite[\textsection 7.1]{KSLeabh}, we denote by $U_q(\frak{sl}_{n+1})$ the Drinfeld--Jimbo quantisation of the universal enveloping algebra of $\frak{sl}_{n+1}$, and by $\mathcal{O}_q(\mathrm{SU}_{n+1})$ the dual quantised coordinate algebra. We then define the \emph{full quantum flag manifold} of $\mathcal{O}_q(\mathrm{SU}_{n+1})$ to be the coideal subalgebra 
$$
\OO_q(\mathrm{F}_{n+1}) := \Big\{ b \in \mathcal{O}_q(\mathrm{SU}_{n+1}) \, | \, K^{\pm 1}_i \triangleright b = b \Big\}. 
$$ 
Just as for $\OO_q(\mathrm{F}_3)$, this is a quantum homogeneous space. Recall next that the Weyl group of $\frak{sl}_{n+1}$ is the symmetric group $S_{n+1}$, and that for any reduced decomposition of $\omega_0$, the longest element of $S_{n+1}$, we have an associated set of root vectors  
$$
\{ X_{\gamma} \,|\, \gamma \in \Delta\} \subseteq U_q(\frak{sl}_{n+1}),
$$
labeled by $\Delta$, the set of roots of $\frak{sl}_{n+1}$. (See \cite[\textsection 6.2]{KSLeabh}, or \cite[Appendix A]{ROBPSLusz}, for a more detailed presentation of Lusztig's root vectors.)

As shown in \cite{ROBPSLusz}, for either of the reduced decompositions  
\begin{align*}
w_0 &=  (s_n s_{n-1} \cdots s_1)(s_n s_{n-1} \cdots s_2) \cdots (s_n s_{n-1})s_n 
\\&=(s_1 s_{2} \cdots s_n)(s_1 s_{2} \cdots s_{n-1}) \cdots (s_1s_{2})s_1
\end{align*}
the associated space of positive Lusztig root vectors, that is the space spanned by the elements $X_{\gamma}$, for $\gamma \in \Delta^+$, is a quantum tangent space $T^{(0,1)}$ for $\OO_q(\mathrm{F}_{n+1})$. For the special case of $\frak{sl}_3$, this reduces to the $\OO_q(\mathrm{F}_3)$ tangent space presented in \textsection \ref{subsection:Lusztig.F3}, and for the special case of the Podle\'s sphere it reduces to anti-holomorphic part of the tangent space of the Podle\'s calculus. 

This quantum tangent space is a direct $q$-deformation of the holomorphic tangent space of the classical full flag manifold, and we denote the associated covariant dc by $\Omega^{(0,1)}$. We denote the basis elements of the cotangent space $V^{(0,1)}$ by $e_{\gamma}$, for $\gamma \in \Delta^+$. 

The space of Lusztig root vectors in fact forms a tangent space for $\OO_q(\mathrm{SU}_{n+1})$. Just as for the rank $2$ case, we can consider $(T^{(0,1)})^*$ the $*$-extension of $T^{(0,1)}$ and then restrict to the full quntum flag manifold. We denote the associated fodc by $\Omega^1_q(\mathrm{F}_{n+1})$ and observe that by construction it admits a direct sum decomposition 
$$
\Omega^1_q(\mathrm{F}_{n+1}) \cong \Omega^{(1,0)}\oplus \Omega^{(0,1)},
$$
where $\Omega^{(1,0)}$. We denote the basis of the associated cotangent space by $f_{\gamma}$, for $\gamma \in \Delta^+$. We now establish a direct generalisation of Proposition \ref{prop:almost.complex.class} to this higher rank setting, with a sketched proof.

\begin{prop}
For the dc $\Omega^1_q(\mathrm{F}_{n+1})$, there exist, up to identification of opposite structures, $2^{|\Pi|}$ left $\mathcal{O}_q(\mathrm{SU}_{n+1})$-covariant foacs. 
\end{prop}
\begin{proof} (Sketch)
The proof is a direct extension of the proof for $\OO_q(\mathrm{F}_3)$. The right $\OO_q(\mathrm{SU}_{n+1})$-module structure of $V^{(0,1)}$ is given explicitly in \cite[Proposition 3.7]{ROBPSLusz}. The right  $\OO_q(\mathrm{SU}_{n+1})$-module structure of $V^{(1,0)}$ can then be concluded from 
\begin{align*}
[S(\omega^*)]b = [S(\omega S^{-1}(b^*))].
\end{align*}
In short, this implies that 
\begin{align*}
V^{(1,0)} \cong \bigoplus_{\gamma \in \Pi} f_{\gamma}\OO_q(\mathrm{SU}_{n+1}), & & V^{(0,1)} \cong \bigoplus_{\gamma \in \Pi} e_{\gamma}\OO_q(\mathrm{SU}_{n+1}). 
\end{align*}
It now follows that, just as for $\OO_q(\mathrm{F}_3)$, a covariant foacs on the dc is determined by assigning to the basis elements $e_{\gamma}$ and $f_{\gamma}$, for $\gamma$ a simple root, the label of holomorphic or anti-holomorphic. Moreover, any such assignment necessarily gives a decomposition of right $\OO_q(\mathrm{F}_{n+1})$-modules
$$
V^1 \cong V^{(1,0)} \oplus V^{(0,1)}.
$$ 
This allows us to appeal to Lemma \ref{lem:foacs}. Considering $V^1$ as a subspace of $\Lambda^1$, the cotangent space of the fodc on $\OO_q(\mathrm{SU}_{n+1})$.

Noting that by construction $e^*_{\gamma} = f_{\gamma}$, for all $\gamma \in \Delta^+$, we now see that the only decompositions that give foacs are those two decompositions where $e_{\gamma}$ and $f_{\gamma}$, for $\gamma \in \Pi$, are contained in complementary summands of the decomposition. Thus we see that we have $2^{|\Pi|}$ covariant foacs for the dc.
\end{proof}

The relations of the maximal prolongation of  $\Omega^1_q(\mathrm{F}_{n+1})$ have not, as of now, been calculated. However, we expect that the results for the $\OO_q(\mathrm{F}_3)$-case extend directly. This is formally presented in the following conjecture.

\begin{conj}
For each full quantum flag manifold $\OO_q(\mathrm{F}_{n+1})$, endowed with the fodc $\Omega^1_q(\mathrm{F}_{n+1})$, each of its $2^{|\Pi|}$ left $\OO_q(\mathrm{SU}_{n+1})$-covariant structures extends to a factorisable, integrable, left $\OO_q(\mathrm{SU}_n)$-covariant almost complex structure on $\Omega^{\bullet}_q(\mathrm{F}_{n+1})$, the maximal prolongation of  $\Omega^{1}_q(\mathrm{F}_{n+1})$.
\end{conj}

Whether this conjecture is true or not, we note that we still have a much smaller number of covariant complex structures than in the classical case, with an upper bound being expressed in terms of the number of simple roots, as opposed to the classical case where it is the number positive roots. Moreover, this conjecture claims that non-integrable almost-complex structures for the full flags are a classical phenomenon.


\section{The Non-existence of a Covariant K\"ahler Structure}

Classically, the flag manifolds possess not only a complex structure, but a K\"ahler structure. Indeed, much of the classical K\"ahler geometry of the irreducible flag manifolds carries over to the quantum setting. The notion of a noncommutative K\"ahler structure was introduced in \cite{ROBKahler} to provide a framework in which to describe this $q$-deformed geometry. Moreover, the existence of a K\"ahler structure in general, was shown to imply direct noncommutative generalisations of many classical results of K\"ahler geometry, such as Lefschetz decomposition  and the K\"ahler identities.

It is thus natural to ask if the dc $\Omega^{\bullet}_q(\mathrm{F}_3)$ admits a noncommutative K\"ahler structure, and in particular, if it admits a left $\OO_q(\mathrm{SU}_3)$-covariant K\"ahler structure. The definition \cite[Definition 7.1]{ROBKahler} of a noncommutative K\"ahler structure requires a central element of the algebra $\Omega^{\bullet}_q(\mathrm{F}_3)$ that is \emph{non-degenerate}, that is to say, a form satisfying $\kappa^3 \neq 0$. Moreover, if the K\"ahler structure is covariant, then $\kappa$ must be a left $\OO_q(\mathrm{F}_3)$-coinvariant element. We will prove the non-existence of a covariant K\"ahler structure by showing that no such form $\kappa$ exists.

\subsection{The Classical K\"ahler Geometry of $\mathrm{F}_3$}

Each almost-complex structure comes with a distinguished homogeneous Riemannian metric. In the integrable case these metrics are permuted by the Weyl group and all three are K\"ahler--Einstein. For the non-integrable case there is a choice of nearly-K\"ahler metric, recalling the special role $6$-dimensional manifolds play in the theory of nearly-K\"ahler geometry \cite{Nagy.Nearly.Kahler}. We now present this information in the form of a table.

\begin{table}
\label{tab.complex.structures}
  \centering
  \begin{tabular}{|c|c|c|}
    \hline
            &  \\
 $T^{(0,1)}$ &  Metric \\
          &  \\
    \hline
            &  \\
    $\Big\{E_1,\, E_2, \, [E_1,E_2] \Big\}$ &  K\"ahler--Einstein  \\
     & \\
    $\Big\{F_1,\, E_2, \, [E_1,E_2] \Big\}$   & K\"ahler--Einstein  \\
         & \\
    $\Big\{E_1,\, F_2, \, [E_1,E_2] \Big\}$ & K\"ahler--Einstein  \\
         & \\
    $\Big\{E_1,\, E_2, \, [F_1,F_2] \Big\}$  &  Nearly K\"ahler \\
             & \\
    \hline
  \end{tabular}
  \label{tab:complex.structures}
\end{table}

An interesting observation is that the classical nearly K\"ahler structure of $\mathrm{F}_3$ is associated to an almost complex structure that does not extend to the quantum setting. Thus we do not have a quantum nearly K\"ahler structure. 

\subsection{Some General Observations}

We begin with two simple general lemmas, which are undoubtedly well-known to the experts, but in the absence of a reference, we include for the reader's convenience.

\begin{lem}\label{lem:coinvariants}
Let $B \subseteq A$ be a quantum homogeneous space and $\mathcal{F} \in {}^A_B\mathrm{Mod}$ a relative Hopf module. Then it holds that 
\begin{align*}
{}^{\mathrm{co}(A)}\Big(1 \square_{\pi} \Phi(\mathcal{F})\Big) = 1 \otimes \big({}^{\pi_B(A)}\Phi(\mathcal{F})\big).
\end{align*} 
\end{lem}
\begin{proof}
Note first that since the only left $A$-coinvariant elements of $A$ are scalar multiples of the identity of $A$, we have that the space of left $A$-coinvariant elements of $A \otimes \Phi(\FF)$ is given by $1 \otimes \Phi(\FF)$. Moreover, we observe that an element of $1 \otimes ({}^{\pi_B(A)}\Phi(\mathcal{F}))$ is contained in the cotensor product $A \square_{\pi_B} \Phi(\mathcal{F})$ if and only if it is an element of $1 \otimes ({}^{\pi_B(A)}\Phi(\mathcal{F}))$. Thus the claimed equality follows.
\end{proof}

\begin{lem}
Let $f \in \mathcal{F}$ be a left $A$-coinvariant element. Then $fb = bf$, for all $b \in B$, if and only if $[f]b = \e(b)[f]$, for all $b \in B$.
\end{lem}
\begin{proof}
Since $f$ is left $A$-coinvariant by assumption, it holds that  $\unit(f) = 1 \otimes [f]$. Now if $[f]b = \e(b)[f]$, for all $b \in B$, then 
$$
(1 \otimes [f])b = b_{(1)} \otimes [f]b_{(2)} = b_{(1)} \otimes \e(b_{(2)})[f] = b \otimes [f] = b(1 \otimes [f]).
$$
Thus since $\unit$ is a $B$-bimodule map, we see that $fb = bf$.

In the other direction, if $fb = bf$ then we see that 
\begin{align*}
[f]b = [f b] = [bf] = \e(b)[f], & & \textrm{ for all } b \in B.
\end{align*}
This establishes the claimed equivalence.
\end{proof}

\subsection{The Differential Calculus $\Omega^1_q(\mathrm{F}_3)$}

We now apply the results of the previous subsection to the dc $\Omega^{\bullet}_q(\mathrm{F})$, describing first the space of degree-$2$ coinvariant forms. 

\begin{lem} \label{lem:leftcov}
The space of left $\O_q(\mathrm{SU}_3)$-coinvariant $2$-forms is a three-dimensional space. The corresponding space of $\OO_q(\mathbb{T}^2)$-coinvariant degree-$2$ elements of $V^{\bullet}$ is spanned by 
\begin{align*}
f_{\alpha_1} \wedge e_{\alpha_1}, & & f_{\alpha_2} \wedge e_{\alpha_2}, & & f_{\alpha_1+\alpha_2} \wedge e_{\alpha_1+\alpha_2}.   
\end{align*}
\end{lem}
\begin{proof}
The left $\OO(\mathbb{T}^2)$-coinvariant elements of $V^{\bullet}$ are simply the elements of weight zero. Looking at the degree two elements of $V^{\bullet}$, and consulting Table \ref{table:sumofroots}, we see that the only weight zero basis elements are $f_{\alpha_1} \wedge e_{\alpha_1}$, $f_{\alpha_2} \wedge e_{\alpha_2}$, and $f_{\alpha_1+\alpha_2} \wedge e_{\alpha_1+\alpha_2}$. Thus we see that the space of coinvariants is three-dimensional as claimed.  
\end{proof}

\begin{lem} \label{lem:F3centralforms}
The space of  left $\OO_q(\mathbb{T}^2)$-coinvariant elements $v \in V^2$ satisfying $vb = \e(b)v$, for all $b \in \OO_q(\mathrm{F}_2)$, is spanned by the elements
\begin{align*}
f_{\alpha_2} \wedge e_{\alpha_2}, & & f_{\alpha_1+\alpha_2} \wedge e_{\alpha_1+\alpha_2}.   
\end{align*}
\end{lem}
\begin{proof}
Note first that 
$$
(f_{\alpha_2} \wedge e_{\alpha_2}) b = f_{\alpha_2}b_{(1)} \wedge e_{\alpha_2}b_{(2)} = f_{\alpha_2}b_{(1)} \wedge e_{\alpha_2}\e(b_{(2)}) = \e(b) f_{\alpha_2} \wedge e_{\alpha_2},
$$
with the analogous result for $f_{\alpha_1+\alpha_2} \wedge e_{\alpha_1+\alpha_2}$. Let us now look at the remaining $\mathcal{O}_q(\mathbb{T}^2)$-coinvariant basis element $f_{\alpha_2} \wedge e_{\alpha_2}$ . It follows from \eqref{eqn:grading.on.gens} that the element $u_{11}u_{32}u_{33}$ is contained in $\OO_q(\mathrm{F}_3)$. We see that 
\begin{align*}
(f_{\alpha_1} \wedge e_{\alpha_1})u_{11}u_{32}u_{23} = & \, \sum_{a,b,c=1}^3 (f_{\alpha_1}u_{1a}u_{3b}u_{3c}) \wedge (e_{\alpha_1}u_{a1}u_{b2}u_{c3}) \\
                                                    = & \, \sum_{a,b,c=1}^3 (f_{\alpha_1}u_{11}u_{23}u_{33}) \wedge (e_{\alpha_1}u_{11}u_{32}u_{33}) \\
                                                    = & -q^{-3}\nu^2 f_{\alpha_1} \wedge e_{\alpha_1}\\
                                                    \neq & ~~ 0.
\end{align*}
Thus $(f_{\alpha_1} \wedge e_{\alpha_1})u_{11}u_{23}u_{33}  \neq \e(b)f_{\alpha_1} \wedge e_{\alpha_1}$, for all $b$, meaning that the space of coinvariant forms, whose right $\OO_q(\mathrm{F_3})$-action is trivial, is two-dimensional as claimed. \cite{ROBKahler}
\end{proof}

\begin{lem}\label{lem:nondegenerate}
For an arbitrary left $\OO_q(\mathrm{SU}_3)$-coinvariant form
$$
\omega := c_1 f_{\alpha_1} \wedge e_{\alpha_1} + c_2 f_{\alpha_2} \wedge e_{\alpha_2} + c_3  f_{\alpha_1+\alpha_2} \wedge e_{\alpha_1+\alpha_2}
$$
it holds that $\omega^3 \neq 0$ only if $c_1 \neq 0$.
\end{lem}
\begin{proof}
Let us consider the case of $c_1 = 0$ and consider the third power of the form. This will be a linear combination of elements of the form
\begin{align*}
    c_ic_jc_k f_{\alpha_i} \wedge e_{\alpha_i} \wedge f_{\alpha_j} \wedge e_{\alpha_j} \wedge f_{\alpha_k} \wedge e_{\alpha_k},
\end{align*}
for $i,j,k = 2,3$, and for convenience we have denoted $\alpha_3 = \alpha_1 + \alpha_2$. It follows directly from the commutation relations given in \textsection \ref{thm:thereal} that all such products are zero. For example, we see that the product 
$$
f_{\alpha_2} \wedge e_{\alpha_2} \wedge f_{\alpha_2} \wedge e_{\alpha_2} \wedge f_{\alpha_3} \wedge e_{\alpha_3}
$$
is equal to 
\begin{align*}
 & \, - q^{-2} f_{\alpha_2} \wedge  f_{\alpha_2} \wedge e_{\alpha_2} \wedge e_{\alpha_2} \wedge f_{\alpha_3} \wedge e_{\alpha_3} + \nu f_{\alpha_2} \wedge  f_{\alpha_3} \wedge e_{\alpha_3} \wedge e_{\alpha_3} \wedge f_{\alpha_3} \wedge e_{\alpha_3}
\end{align*}
which is in turn equal to 
\begin{align*}
q^{(\alpha_1+\alpha_2,\alpha_1+\alpha_2)} \nu f_{\alpha_2} \wedge  f_{\alpha_3}  \wedge  f_{\alpha_3} \wedge e_{\alpha_3} \wedge e_{\alpha_3} \wedge e_{\alpha_3} = 0.
\end{align*}
Thus the form $\omega^3$ is equal to zero as claimed.
\end{proof}

Combining the statements of Lemma \ref{lem:F3centralforms} and Lemma \ref{lem:nondegenerate} we arrive at the following theorem, which means that the dc $\Omega^{\bullet}_q(\mathrm{F}_3)$ does not admit a covariant K\"ahler structure.

\begin{thm}
There does not exist a left $\OO_q(\mathrm{SU}_3)$-coinvariant non-degenerate form $\sigma \in \Omega^2_q(\mathrm{F}_3)$ that commutes with the elements of $\OO_q(\mathrm{F}_3)$.
\end{thm}

\begin{remark} 
Despite the fact that any nondegenerate coinvariant $2$-form does not commutate with the elements of $\OO_q(\mathrm{F}_3)$, we can still define left and right Lefschetz maps. As is readily checked, each map is (either a left or right) $\OO_q(\mathrm{F}_3)$-module isomorphism. Each has an associated Lefschetz decomposition with a corresponding Hodge map, metric, and inner product.
\end{remark}

\begin{remark}
The nonexistence of a coinvariant non-degenerate form also implies that $\Omega^2_q(\mathrm{F}_3)$ does not admit a metric in the sense of Beggs and Majid \cite{BeggsMajid:Leabh}. This implies that $\Omega^1_q(\mathrm{F}_3)$ is not self-dual as an object in the category of relative Hopf modules ${}^A_B\mathrm{Mod}_B$, as explained for example in \cite{LC.Praha.Kolkata}.
\end{remark}

\begin{remark}
The definition of a K\"ahler structure also requires that the K\"ahler form $\kappa$ is \emph{real}, that is to say $\kappa^* = \kappa$, and \emph{closed}, that is to say $\exd \kappa = 0$. A family of real, closed, left $\OO_q(\mathrm{SU}_3)$-coinvariant $2$-forms can be constructed from the K\"ahler structures of the two copies of $\OO_q(\mathbb{CP}^2)$ in $\OO_q(\mathrm{F}_3)$. 

As shown in \textsection \ref{subsection:HKrestriction}, the Heckenberger--Kolb double complex of each copy of $\OO_q(\mathbb{CP}^2)$ is realised as the restriction of the $*$-dc $\Omega^{\bullet}_q(\mathrm{F}_3)$. Thus the K\"ahler forms $\kappa_1$ and $\kappa_2$ of these dc, as introduced in \cite{ROBKahler}, will be real closed left $\OO_q(\mathrm{SU}_3)$-coinvariant elements of $\Omega^{2}_q(\mathrm{F}_3)$. Thus we see that the $2$-form 
\begin{align*}
\kappa_1 + \lambda \kappa_2, & & \textrm{ for } \lambda \in \mathbb{C}^{\times}
\end{align*}
is a real closed left $\OO_q(\mathrm{SU}_3)$-coinvariant $2$-form. Classically this form is the fundamental form of a K\"ahler metric for $\mathrm{F}_3$.
\end{remark}


\appendix 

\section{Maximal Prolongations and Framing Calculi} \label{app:MPFC}

In this appendix, we present an explicit formula for the degree two generators of the maximal prolongation of a left $A$-covariant fodc over a quantum homogeneous space $B \subseteq A$, in the presence of a framing calculus $\Omega^1(A) \cong A \otimes \Lambda^1$. This formula is well known for the special case of a Hopf algebra $A = B$ \cite[\textsection 14.3]{KSLeabh}. Moreover, an alternative version appeared in \cite[\textsection 5]{MMF2} under the assumption that $\Omega^1(B)$ satisfies the additional property $\Omega^1(B)B^+ = B^+\Omega^1(B)$. Here we do not make this addition assumption.

For any algebra $B$, the \emph{universal fodc} is the fodc $(\Omega^1_u(B),\exd_u)$, where $\Omega^1_u(B)$ is the kernel of the multiplication map $m: B \otimes B \to B$, and $\exd_u b = 1 \otimes b - b \otimes 1$. Every fodc is isomorphic to a quotient $\Omega^1_u(B)/N$, for some sub-bimodule $N \subseteq \Omega^1_u(B)$. Moreover, this gives a bijective correspondence between fodc and sub-bimodules. 
Now the maximal prolongation of $\Omega^1(B)$ can be explicitly described as the quotient of the tensor algebra of $\Omega^1(B)$ by the two-sided ideal generated by the set of elements
\begin{align*}
  N^{(2)} :=  \Big\{\delta(\sum_i b_i\exd_u c_i) := \sum_i \exd b_i \otimes \exd c_i \,|\, \textrm{ for } \sum_i b_i\exd c_i \in N\}.
  \end{align*}

For a Hopf algebra $A$, the fundamental theorem of Hopf modules is a monoidal equivalence between ${}^A_A\mathrm{Mod}_A$ and $\mathrm{Mod}_A$, with respect to the obvious monoidal structure of each category. Explicitly, the components of the equivalence are given by 
\begin{align*}
\mu_{\MM,\NN}: F(\MM \otimes_A \NN) \to F(\MM) \otimes F(\NN), & & [m \otimes n] \mapsto [mn_{(-1)}] \otimes [n_{(0)}].
\end{align*}
For a left $A$-covariant fodc $\Omega^1(A)$ over $A$, this gives us an isomorphism
\begin{align*}
F(\Omega^1(A)^{\otimes_A k}) \to F(\Omega^1(A))^{\otimes k}. 
\end{align*}
Restricting this isomorphism to the image of the tensor algebra of $\Omega^1(B)$ under $\Phi$ we get an isomorphism, in the category ${}^A_B\mathrm{Mod}_B$, with the tensor algebra of $V^{1}$. 

Denoting by $V^{\bullet}$ the quotient of the tensor algebra of $V^1$ by the ideal generated by the subspace
$$
I^{(2)} := \mu(\Phi(N^{(2)})).
$$
We now see that the maximal prolongation $\Omega^{\bullet}(B)$ is isomorphic $A \square_{\pi_B} V^{\bullet}$, as described in \textsection \ref{subsection:remarksQHTS}. The following proposition now gives an explicit presentation of $I^{(2)}$ in terms of the elements of the ideal $I$.

\begin{prop}
It holds that 
\begin{align} \label{eqn:I(2)}
I^{(2)} = \Big\{ [\exd y_{(1)}] \otimes [\exd y_{(2)}] \,|\, y \in I \Big\}.
\end{align}
\end{prop}
\begin{proof}
Note first that, for $I$ the ideal of $B^+$ corresponding to $\Omega^1(B)$, and $\sum_i a_i \otimes [y_i]$ an element of $A \square_{\pi_B} I$, it holds that 
\begin{align*}
\unit^{-1}(\sum_i a_i \otimes [y_i]) = & \, \sum_i a_i S((y_i)_{(1)})\exd (y_i)_{(2)}.
\end{align*}
Thus it holds that 
\begin{align*}
 N^{(1)} = \Big\{\sum_i a_i S((y_i)_{(1)})\exd (y_i)_{(2)} \,|\, \sum_i a_i \otimes y_i \in A \square_{\pi_B} I\Big\}.  
\end{align*}
From this it follows that
\begin{align*}
 N^{(2)} = \delta N^{(1)} = \Big\{\sum_i \exd(a_i S((y_i)_{(2)})) \otimes \exd (y_i)_{(2)} \,|\, \sum_i a_i \otimes y_i \in A \square_{\pi_B} I\Big\}.   
\end{align*}
Taking the image of this object in ${}^A_B \mathrm{Mod}_B$ under the functor $\Phi$ we get the following object in ${}^{\pi_B} \mathrm{Mod}$
\begin{align*}
 \Big\{\sum_i [\exd(a_i S((y_i)_{(1)})) \otimes \exd (y_i)_{(2)}] \,|\, \sum_i a_i \otimes y_i \in A \square_{\pi} I\Big\}.  
\end{align*}
Operating on this by the restriction of $\mu$, we get the subspace $I^{(2)}$. So let us next consider the action of $\mu$ on an element of $F(\Omega^1(A) \otimes_A \Omega^1(A))$ of the form $[\exd aS(c_{(1)}) \otimes \exd c_{(2)}]$:
\begin{align*}
\mu\Big([\exd(aS(y_{(1)})) \otimes \exd y_{(2)}]\Big) = & \, [\exd(a S(y_{(1)}))y_{(2)}] \otimes [\exd y_{(3)}].
\end{align*}
An application of the Leibniz rule, and the antipode axiom, then yields the expression
\begin{align*}
[\exd a] \otimes [\exd y]  + \e(a)[\exd(S(y_{(1)}))y_{(2)}] \otimes [\exd y_{(3)}]
\end{align*}
Now since $\exd y \in I$, its coset is trivial, and so, the first summand goes to zero, giving
\begin{align*}
\e(a)[\exd(S(y_{(1)}))y_{(2)}] \otimes [\exd y_{(3)}]
\end{align*}
Noting next that, for any $a \in A$, we have that $\exd(S(a_{(1)}))a_{(2)} = - S(a_{(1)})\exd a_{(2)}$, our expression can be reduced to 
\begin{align*}
- \e(a)[\exd y_{(1)}] \otimes [\exd y_{(2)}]
\end{align*}
Thus we see that $I^{(2)}$ is contained in the right hand side of \eqref{eqn:I(2)}. To show the opposite inclusion, consider an element $[\exd z] \in I$ and the inverse of the counit map
$$
\counit^{-1}: I \to \Phi(A \square_{\pi_B} I). 
$$
Using the standard decomposition $A \cong A^+ \otimes \mathbb{C}1$, we see that $\counit^{-1}(\exd z)$ can be written as a sum $1 \otimes \exd z + \sum_i a_i \otimes \exd z_i$, where $a_i \in A^+$. The corresponding element in $I^{(2)}$ will, of course, be $[\exd z_{(1)}] \otimes [\exd z_{(2)}]$, giving us the opposite inclusion.
\end{proof}

Finally, we recall that an isomorphism $\sigma$, in the category ${}^A\mathrm{Mod}_B$ between $\Phi(\Omega^1_u(B))$ and $B^+$ is given by $\sigma[\exd b] := b^+ := b - \e(B)1$. This means that $\Phi(\Omega^1(B))$ is isomorphic to $V^1_B := B^+/I$, where by abuse of notation we have denoted $sigma(I)$ by $I$. This in turn means that the tensor algebra of $\Phi(\Omega^1(B))$ is isomorphic to the tensor algebra of $V^1_B$. Finally, we see that this gives us the quantum exterior algebra presentation of the maximal prolongation given in \textsection \ref{subsection:remarksQHTS}, where we have again, by abuse of notation, denoted $\sigma(I^{(2)})$ by $I^{(2)}$.

\section{Some Details on the Lie Algebra $\mathfrak{sl}_3$} \label{app:sl3}

In this subsection, so as to set the notation, we recall some elementary definitions and results about the $A_2$ root system associated to the special linear Lie algebra $\mathfrak{sl}_{3}$. Let $\{\e_i\}_{i=1}^{n+1}$ be the standard basis of $\mathbb{R}^{3}$, endowed with its usual Euclidean structure.  The root system $A_{2}$ is the pair $(V,\Delta)$, where $V$ is the subspace of $\mathbb{R}^{3}$ spanned by the roots
$$
\Delta := \Big\{ \pm \alpha_1 := \pm(\e_1-\e_2), \, \pm \alpha_2 := \pm(\e_2-\e_3), \, \pm (\alpha_1 + \alpha_2) = \pm(\e_1 - \e_3) \Big\}.
$$
We take the standard subset of positive roots, and its associated set of simple roots, 
\begin{align*}
\Delta^+ := \Big\{ \alpha_1, \, \alpha_2, \, \alpha_1 + \alpha_2\Big\}, & & \Pi := \Big\{\alpha_1, \, \alpha_2 \Big\}.
\end{align*}
The Weyl group of the root system is the symmetric group $S_{3}$ of order $6$, and the Cartan matrix is
$$
(a_{ij})_{ij} = 
\begin{pmatrix}
2 & -1  \\
-1 & 2 
\end{pmatrix}
$$

We finish with a table presenting all possible sums $\alpha + \beta$, where $\alpha,\beta \in \Delta^+$. The sums highlighted in blue are those that again roots of $\frak{sl}_3$. We appeal to this table a number of times in the paper. For example, we refer to it when classifying the left coinvariant forms in Lemma \ref{lem:leftcov}.

\begin{table}
\caption{Sums of roots of $\frak{sl}_3$} \label{table:sumofroots}
\begin{small}
\begin{center}
\begin{tabular}{|c|c|c|c|c|c|c|}
    \hline
   & $\alpha_1$ & $\alpha_2$ & $\alpha_1+\alpha_2$ & $-\alpha_1$ & $-\alpha_2$ & $-(\alpha_1+\alpha_2)$\\
    \hline
    $\alpha_1$ &$2\alpha_1$&  {\color{blue}$\alpha_1+\alpha_2$} & $2\alpha_1+\alpha_2$ & $0$ & $\alpha_1-\alpha_2$ &   {\color{blue}$-\alpha_2$} \\
    \hline
    $\alpha_2$ &  {\color{blue}$\alpha_1+\alpha_2$}& $2\alpha_2$ & $\alpha_1+2\alpha_2$ & $\alpha_2-\alpha_1$ & $0$ &   {\color{blue}$-\alpha_1$} \\
    \hline
   $\alpha_1+\alpha_2$ & $2\alpha_1+\alpha_2$& $\alpha_1+2\alpha_2$ & $2(\alpha_1+\alpha_2)$ &  {\color{blue}$\alpha_2$} &  {\color{blue}$\alpha_1$ }&  $0$\\
    \hline
    $-\alpha_1$& $0$& $\alpha_2-\alpha_1$ &  {\color{blue}$\alpha_2$} & $-2\alpha_1$ &  {\color{blue}$-(\alpha_1+\alpha_2)$} &  $-(2\alpha_1+\alpha_2)$  \\
    \hline
   $-\alpha_2$ & $\alpha_1-\alpha_2$& $0$ &  {\color{blue}$\alpha_1$} &  {\color{blue}$-(\alpha_1+\alpha_2)$} & $-2\alpha_2$ &  $-(\alpha_1+2\alpha_2)$ \\
    \hline
    $-(\alpha_1+\alpha_2)$ &  {\color{blue}$-\alpha_2$}&  {\color{blue}$-\alpha_1$} & $0$ & $-(2\alpha_1+\alpha_2)$ & $-(\alpha_1+2\alpha_2)$ &  $-2(\alpha_1+\alpha_2)$\\
    \hline
\end{tabular}
\end{center}
\end{small}
\end{table}

\end{document}